\documentclass[11pt]{amsart}

\usepackage[all]{xy}
\usepackage{color, pstricks}
\usepackage{mathrsfs}
\usepackage{amsmath,amsthm, amssymb, amsgen,amsxtra,amsfonts, amsbsy} 
\usepackage[all]{xy}
\usepackage{verbatim}
\usepackage{url}
\usepackage{hyperref}
\usepackage{bbm}
\usepackage{bbold}

\usepackage{array,multirow,graphicx}

\DeclareMathOperator{\Rep}{Rep}

\begin{document}
%
%
%
\theoremstyle{definition}
\newtheorem{Definition}{Definition}[section]
\newtheorem*{Definitionx}{Definition}
\newtheorem{Convention}[Definition]{Convention}
\newtheorem{Construction}{Construction}[section]
\newtheorem{Example}[Definition]{Example}
\newtheorem{Examples}[Definition]{Examples}
\newtheorem{Remark}[Definition]{Remark}
\newtheorem*{Remarkx}{Remark}
\newtheorem{Remarks}[Definition]{Remarks}
\newtheorem{Fact}[Definition]{Fact}
\newtheorem{Facts}[Definition]{Facts}
\newtheorem{Caution}[Definition]{Caution}
\newtheorem{Conjecture}[Definition]{Conjecture}
\newtheorem*{Conjecturex}{Conjecture}
\newtheorem{Question}[Definition]{Question}
\newtheorem{Questions}[Definition]{Questions}
\newtheorem*{Acknowledgements}{Acknowledgements}
\newtheorem*{Organization}{Organization}
\newtheorem*{Disclaimer}{Disclaimer}
\theoremstyle{plain}
\newtheorem{Theorem}[Definition]{Theorem}
\newtheorem*{Theoremx}{Theorem}
\newtheorem{Theoremy}{Theorem}
\newtheorem{Proposition}[Definition]{Proposition}
\newtheorem*{Propositionx}{Proposition}
\newtheorem{Lemma}[Definition]{Lemma}
\newtheorem{Corollary}[Definition]{Corollary}
\newtheorem*{Corollaryx}{Corollary}
\newtheoremstyle{voiditstyle}{3pt}{3pt}{\itshape}{\parindent}%
{\bfseries}{.}{ }{\thmnote{#3}}%
\theoremstyle{voiditstyle}
\newtheorem*{VoidItalic}{}
\newtheoremstyle{voidromstyle}{3pt}{3pt}{\rm}{\parindent}%
{\bfseries}{.}{ }{\thmnote{#3}}%
\theoremstyle{voidromstyle}
\newtheorem*{VoidRoman}{}

%
\newcommand{\prf}{\par\noindent{\sc Proof.}\quad}
\newcommand{\blowup}{\rule[-3mm]{0mm}{0mm}}
\newcommand{\cal}{\mathcal}
\newcommand{\Aff}{{\mathds{A}}}
\newcommand{\BB}{{\mathds{B}}}
\newcommand{\CC}{{\mathds{C}}}
\newcommand{\EE}{{\mathds{E}}}
\newcommand{\FF}{{\mathds{F}}}
\newcommand{\GG}{{\mathds{G}}}
\newcommand{\HH}{{\mathds{H}}}
\newcommand{\NN}{{\mathds{N}}}
\newcommand{\ZZ}{{\mathds{Z}}}
\newcommand{\PP}{{\mathds{P}}}
\newcommand{\QQ}{{\mathds{Q}}}
\newcommand{\RR}{{\mathds{R}}}
\newcommand{\Liea}{{\mathfrak a}}
\newcommand{\Lieb}{{\mathfrak b}}
\newcommand{\Lieg}{{\mathfrak g}}
\newcommand{\Liem}{{\mathfrak m}}
\newcommand{\ideala}{{\mathfrak a}}
\newcommand{\idealb}{{\mathfrak b}}
\newcommand{\idealg}{{\mathfrak g}}
\newcommand{\idealm}{{\mathfrak m}}
\newcommand{\idealn}{{\mathfrak n}}
\newcommand{\idealp}{{\mathfrak p}}
\newcommand{\idealq}{{\mathfrak q}}
\newcommand{\idealI}{{\cal I}}
\newcommand{\lin}{\sim}
\newcommand{\num}{\equiv}
\newcommand{\dual}{\ast}
\newcommand{\iso}{\cong}
\newcommand{\homeo}{\approx}
\newcommand{\mathds}[1]{{\mathbb #1}}
\newcommand{\mm}{{\mathfrak m}}
\newcommand{\pp}{{\mathfrak p}}
\newcommand{\qq}{{\mathfrak q}}
\newcommand{\rr}{{\mathfrak r}}
\newcommand{\pP}{{\mathfrak P}}
\newcommand{\qQ}{{\mathfrak Q}}
\newcommand{\rR}{{\mathfrak R}}
%
%
\newcommand{\OO}{{\cal O}}
\newcommand{\calA}{{\cal A}}
\newcommand{\calG}{{\cal G}}
\newcommand{\calI}{{\cal I}}
\newcommand{\calM}{{\cal M}}
\newcommand{\calO}{{\cal O}}
\newcommand{\calU}{{\cal U}}
\newcommand{\calV}{{\cal V}}
\newcommand{\calH}{{\cal H}}
\newcommand{\shHom}{{\calH om}}
\newcommand{\numero}{{n$^{\rm o}\:$}}
\newcommand{\mf}[1]{\mathfrak{#1}}
\newcommand{\mc}[1]{\mathcal{#1}}
\newcommand{\into}{{\hookrightarrow}}
\newcommand{\onto}{{\twoheadrightarrow}}
\newcommand{\Spec}{\mathrm{Spec}\:}
\newcommand{\BigSpec}{{\rm\bf Spec}\:}
\newcommand{\Spf}{\mathrm{Spf}\:}
\newcommand{\Proj}{\mathrm{Proj}\:}
\newcommand{\Pic}{\mathrm{Pic }}
\newcommand{\Picloc}{\mathrm{Picloc }}
\newcommand{\Picloclocloc}{\Picloc^{\loc,\loc}}
\newcommand{\Br}{\mathrm{Br}}
\newcommand{\NS}{\mathrm{NS}}
\newcommand{\id}{\mathrm{id}}
\newcommand{\Sym}{{\mathfrak S}}
\newcommand{\Aut}{\mathrm{Aut}}
\newcommand{\Autp}{\mathrm{Aut}^p}
\newcommand{\End}{\mathrm{End}}
\newcommand{\Hom}{\mathrm{Hom}}
\newcommand{\Ext}{\mathrm{Ext}}
\newcommand{\ord}{\mathrm{ord}}
\newcommand{\Hilb}{\mathrm{Hilb}}
\newcommand{\coker}{\mathrm{coker}\,}
\newcommand{\GHilb}[1]{\mathop{#1\mathchar`-\mathrm{Hilb}}}
\newcommand{\divisor}{\mathrm{div}}
\newcommand{\Def}{\mathrm{Def}}
\newcommand{\et}{\mathrm{\acute{e}t}}
\newcommand{\loc}{\mathrm{loc}}
\newcommand{\ab}{\mathrm{ab}}
\newcommand{\pitop}{{\pi_1^{\mathrm{top}}}}
\newcommand{\pitoploc}{{\pi_{\mathrm{loc}}^{\mathrm{top}}}}
\newcommand{\piet}{{\pi_1^{\mathrm{\acute{e}t}}}}
\newcommand{\pietloc}{{\pi_{\mathrm{loc}}^{\mathrm{\acute{e}t}}}}
\newcommand{\piN}{{\pi^{\mathrm{N}}_1}}
\newcommand{\piNloc}{{\pi_{\mathrm{loc}}^{\mathrm{N}}}}
\newcommand{\piNlocab}{{\pi_{\mathrm{loc}}^{{\rm N}, {\rm ab}}}}
\newcommand{\Het}[1]{{H_{\mathrm{\acute{e}t}}^{{#1}}}}
\newcommand{\Hfl}[1]{{H_{\mathrm{fl}}^{{#1}}}}
\newcommand{\Hcris}[1]{{H_{\mathrm{cris}}^{{#1}}}}
\newcommand{\HdR}[1]{{H_{\mathrm{dR}}^{{#1}}}}
\newcommand{\hdR}[1]{{h_{\mathrm{dR}}^{{#1}}}}
\newcommand{\Torsloc}{\mathrm{Tors}_{\mathrm{loc}}}
\newcommand{\defin}[1]{{\bf #1}}
\newcommand{\oX}{\mathcal{X}}
\newcommand{\oA}{\mathcal{A}}
\newcommand{\oY}{\mathcal{Y}}
\newcommand{\calC}{{\mathcal{C}}}
\newcommand{\calD}{{\mathcal{D}}}
\newcommand{\calF}{{\mathcal{F}}}
\newcommand{\calL}{{\mathcal{L}}}
\newcommand{\calZ}{{\mathcal{Z}}}
\newcommand{\bmu}{\boldsymbol{\mu}}
\newcommand{\balpha}{\boldsymbol{\alpha}}
\newcommand{\bG}{{\mathbf{G}}}
\newcommand{\bL}{{\mathbf{L}}}
\newcommand{\bM}{{\mathbf{M}}}
\newcommand{\bW}{{\mathbf{W}}}
\newcommand{\bD}{{\mathbf{D}}}
\newcommand{\BD}{{\mathbf{BD}}}
\newcommand{\BT}{{\mathbf{BT}}}
\newcommand{\BI}{{\mathbf{BI}}}
\newcommand{\BO}{{\mathbf{BO}}}
\newcommand{\C}{{\mathbf{C}}}
\newcommand{\Dic}{{\mathbf{Dic}}}
\newcommand{\SL}{{\mathbf{SL}}}
\newcommand{\SU}{{\mathbf{SU}}}
\newcommand{\MC}{{\mathbf{MC}}}
\newcommand{\GL}{{\mathbf{GL}}}
\newcommand{\Tors}{{\mathbf{Tors}}}
\newcommand{\Dieu}{{\mathds{D}}}

\makeatletter
\@namedef{subjclassname@2020}{\textup{2020} Mathematics Subject Classification}
\makeatother

\title[McKay Correspondence]{A McKay Correspondence in positive characteristic}

\author{Christian Liedtke}
\address{TU M\"unchen, Department of Mathematics, Boltzmannstr. 3, 85748 Garching bei M\"unchen, Germany}
\email{christian.liedtke@tum.de}


\begin{abstract}
We establish a McKay correspondence
for finite and linearly reductive subgroup schemes of $\SL_2$ 
in positive characteristic.
As an application, we obtain a McKay correspondence for
\emph{all} rational double point singularities in characteristic $p\geq7$.
 
We discuss linearly reductive quotient singularities and
canonical lifts over the ring of Witt vectors.
In dimension 2, we establish simultaneous resolutions of singularities of 
these canonical lifts via $G$-Hilbert schemes.
  
In the appendix, we discuss several approaches toward the notion of
\emph{conjugacy classes} for finite group schemes:
This is an ingredient in McKay correspondences, but also
of independent interest.
\end{abstract}
\setcounter{tocdepth}{1}
\maketitle
\tableofcontents
\section{Introduction}

\subsection{Klein's classification and McKay's correspondence} 
Felix Klein \cite{Klein} classified
finite subgroups $G$ of $\SL_2(\CC)$:
up to conjugation, there are two infinite series and three isolated cases.

\begin{enumerate}
\item The associated quotient singularity $\CC^2/G$ is called a 
 \emph{Klein singularity} and the singularities arising this way
 are precisely the \emph{rational double point singularities}.
 Its minimal resolution of singularities is a union
 of $\PP^1$'s, whose dual intersection graph $\Gamma$ 
 is a simply-laced Dynkin diagram of finite type, that is,
 of type $A_n$, $D_n$, $E_6$, $E_7$ or $E_8$.
\item John McKay \cite{McKay} associated a finite graph $\widehat{\Gamma}$ 
 to $G \subset\SL_2(\CC)$, whose vertices correspond to the isomorphism classes of the 
 simple representations of $G$.
 This graph is a Dynkin diagram of affine type
 $\widehat{A}_n$, $\widehat{D}_n$, $\widehat{E}_6$, $\widehat{E}_7$ or $\widehat{E}_8$.
\item After these preparations, the \emph{classical McKay correspondence} consists of the 
following observations:
\begin{enumerate}
\item The graph $\Gamma$ is obtained from $\widehat{\Gamma}$ by removing the vertex corresponding to the 
trivial representation.
\item There exists a bijection between 
conjugacy classes of $G$, vertices of $\widehat{\Gamma}$, and isomorphism classes of simple representations of $G$.
\item There exists a bijection between
finite subgroups of $\SL_2(\CC)$ up to conjugacy, the above Dynkin diagrams of affine type,
Klein singularities, and the above Dynkin diagrams of finite type. 
\end{enumerate}
\end{enumerate}

By now, there are various approaches to and a vast literature on this subject, such as 
\cite{Knorrer, Kostant, McKay, McKay2, Steinberg} and many more.
Also, there are now generalisations into very
different directions: 
higher dimensional algebraic geometry \cite{Reid}, K-theory \cite{GSV},
derived categories of coherent sheaves \cite{Bridgeland,KV}, 
representations of quivers \cite{Kirillov}, non-commutative geometry and Hopf algebras \cite{Chan},
and string theory \cite{Dirichlet} - just to mention a few.

\subsection{Positive characteristic}
Now, let $k$ be an algebraically closed field of characteristic $p>0$.

\subsubsection{Wild McKay correspondence}
The classical McKay correspondence as sketched above is still partially 
available over $k$ if $G$ is assumed to be a finite subgroup of $\SL_2(k)$
of order prime to $p$ - this is the \emph{tame} case.
If $p$ divides the order of $G$ - this is the \emph{modular} or \emph{wild} case -
then this correspondence breaks down.
We refer to Yasuda's surveys \cite{Yasuda, YasudaOpen} about conjectures
and partial results concerning such wild McKay correspondences.

\subsubsection{Linearly reductive McKay correspondence}
In this article, we show that if $G$ is a finite and
\emph{linearly reductive} subgroup scheme of $\SL_{2,k}$, 
then there is a reasonable version of the classical McKay correspondence.
For example, we obtain a McKay correspondence for \emph{all} rational double point 
singularities if $p\geq7$.
Instead of considering groups, we allow non-reduced 
group schemes over $k$, but we require their categories of $k$-linear and finite-dimensional
representations to be semi-simple.
We refer to \cite{Yasuda2} about conjectures and partial results concerning
a McKay correspondence for the group scheme $\balpha_p$, which is not linearly reductive.

Thus, let $G$ be a finite and linearly reductive subgroup scheme of $\SL_{2,k}$ 
with $p\geq7$.
(In this introduction, we will exclude small characteristics whenever this is makes
our discussion easier.)
Let $x\in X:=U/G$ with $U=\Aff_k^2$ or $U=\widehat{\Aff}_k^2$
be the associated Klein singularity, 
which is a rational double point.
One goal of this article is to define a notion of conjugacy class for $G$,
to construct graphs $\Gamma$ and $\widehat{\Gamma}$, and  to 
establish bijections as above.

Let us make three comments:
\begin{enumerate}
\item It is interesting in its own that a McKay correspondence can be extended 
from finite group schemes of length prime to $p$, that is, the tame case, 
to linearly reductive group schemes.
\item What makes this linearly reductive McKay correspondence 
really interesting is that the bijection in Theorem \ref{thm: main intro} is \emph{not}
true when considering finite groups of order prime to $p$ only, 
see Example \ref{sec: mup intro}.
\item Probably, many more aspects of the classical McKay correspondence 
can be carried over to the linearly reductive setting, but rather than writing a whole
monograph, we decided to establish only some basic bijections.
\end{enumerate}

\subsection{Linearly reductive group schemes}
Let $G$ be a \emph{finite and linearly reductive group scheme} 
over an algebraically closed field $k$ of characteristic $p\geq0$.
By definition, this means that the category $\mathrm{Rep}_k(G)$ of
$k$-linear and finite-dimensional representations of $G$ is semi-simple.

If $p=0$, then every finite group scheme over $k$ is \'etale and linearly reductive.
and in fact, it is
the constant group scheme associated to a finite group.
In fact, the functor $G\mapsto G_{\mathrm{abs}}:=G(k)$ induces an equivalence 
of categories between finite group schemes over $k$ and finite groups.

If $p>0$, then every linearly reductive group scheme admits a 
canonical semi-direct product decomposition
$$
  G\,\cong\,G^\circ\rtimes G^{\et},
$$
where $G^\et$ is a group scheme of length prime to $p$ (and thus, the constant 
group scheme associated to a finite group of order prime to $p$) and where $G^\circ$
is infinitesimal and diagonalisable.
The latter implies that $G^\circ$ is a product of group schemes of the form
$\bmu_{p^n}$.
Conversely, every such semi-direct product of a diagonalisable group scheme
with the constant group scheme associated to a finite group of order prime to $p$
is linearly reductive.
This structure result is usually attributed to Nagata \cite{Nagata61},
but see also \cite{AOV}, \cite{Chin}, and \cite{Hashimoto}.
In particular, the class of finite and linearly reductive group schemes over $k$
strictly contains the class of constant group schemes associated 
to finite groups of order prime to $p$.

\subsubsection{Abstract groups}
Associated to $G$, there is an abstract finite group $G_{\rm abs}$
and we refer for the slightly technical definition to Section \ref{subsec: abstract groups}.
The order of $G_{\rm abs}$ is equal to the length of $G$.
For example, if $G$ is \'etale over $k$, then we have
$G_{\rm abs}\cong G(k)$.
The assignment $G\mapsto G_{\mathrm{abs}}$ establishes an equivalence
of categories 
\begin{equation}
\label{eq: catequivalence}
\left\{ \begin{array}{l}
\mbox{finite linearly reductive}\\ \mbox{group schemes over $k$} 
\end{array}\right\}
\quad\leftrightarrow\quad
\left\{ \begin{array}{l}
\mbox{finite groups with a}\\ \mbox{unique $p$-Sylow subgroup} 
\end{array}\right\},
\end{equation}
see \cite{LMM} and Lemma \ref{lem: lrequivalence}.

\subsubsection{Canonical lifts}
In \cite{LMM}, we showed that if $G$ is a finite and linearly
reductive group scheme, then there exists a lift of $G$ over the
ring of Witt vectors $W(k)$.
We note that $G^\circ$ and $G^\et$ even lift uniquely to $W(k)$
and we define the \emph{canonical lift} $G_{\rm can}\to \Spec K$ of $G$ to
be the unique lift that is a semi-direct product of the lifts
of $G^\circ$ and $G^{\et}$. 
Any lift of $G$ to some extension field of $K$ becomes
isomorphic to $G_{\mathrm{can}}$ after possibly passing to some
further field extension.
Moreover, the finite group $G_{\mathrm{can}}(\overline{K})$ 
is isomorphic to $G_{\rm abs}$.

\subsubsection{Representation theory}
\label{subsubsec: rep thy}
By \cite{LMM} and Proposition \ref{prop: liftingrepresentation},
there exist canonical equivalences of representation categories 
\begin{equation}
\label{eq: rep thy intro}
   \mathrm{Rep}_{k}(G) \,\to\, \mathrm{Rep}_{\overline{K}}(G_{\mathrm{can},\overline{K}})
   \,\to\, \mathrm{Rep}_\CC(G_{\mathrm{abs}})
\end{equation}
that are compatible with degrees, direct sums, tensor products, duals, 
and simplicity.
These equivalences induce isomorphisms of rings
$$
  K_k(G) \,\to\, K_{\overline{K}}(G_{\mathrm{can},\overline{K}}) \,\to\,
  K_{\CC}(G_{\mathrm{abs}}),
$$
see Corollary \ref{cor: Ktheory}.
Here, $K_F(G)$ denotes the K-group associated to 
$F$-linear and finite-dimensional $G$-representations.

\subsubsection{Hopf algebras}
If $G$ is a finite group scheme over $k$, then the multiplication map
turns $H^0(G,\OO_G)$ into a finite-dimensional \emph{Hopf algebra}
over $k$.
We discuss finite group schemes and among them the linearly reductive
ones from the point of Hopf algebras in 
Appendix \ref{sec: Hopf}.

\subsection{Linearly reductive subgroup schemes of $\SL_{2,k}$}
Let $k$ be an algebraically closed field of characteristic $p\geq0$.
Hashimoto \cite{Hashimoto} extended Klein's classification \cite{Klein}
of finite subgroups of $\SL_2(\CC)$ up to conjugation to the setting
of finite and linearly reductive subgroup schemes of $\SL_{2,k}$.
If $p\geq7$, then one obtains a list analogous to Klein's classical 
list.
If $p\in\{2,3,5\}$, then some classical cases are missing, 
but there are no new cases.

\subsection{McKay graph and McKay correspondence}
\label{subsec: McKay graph intro}
Let $G$ be a finite and linearly reductive subgroup scheme of
$\SL_{2,k}$.
As in McKay's original construction \cite{McKay}, we associate
an affine Dynkin diagram $\widehat{\Gamma}$ to $G$,
its embedding into $\SL_{2,k}$, and the set of isomorphism classes
of simple representations of $G$.
This is the \emph{McKay graph} associated to this data.
In fact, we will see that it is compatible with the equivalences
induced by \eqref{eq: catequivalence} and \eqref{eq: rep thy intro}
and we refer to Section \ref{subsec: McKay graph} for details.
We establish the following version of McKay's theorem \cite{McKay}
in positive characteristic.

\begin{Theorem}[Theorem \ref{thm: main}]
\label{thm: main intro}
 Let $k$ be an algebraically closed field of characteristic $p\geq0$.
 There exists a bijection between
 non-trivial, finite, and linearly reductive subgroup schemes of
 $\SL_{2,k}$ up to conjugation and affine Dynkin graphs
 of type
 $$
  \begin{array}{lcl}
   \widehat{A}_n, \widehat{D}_n, \widehat{E}_6, \widehat{E}_7, \widehat{E}_8
       &\mbox{\quad if \quad }&    p=0\mbox{\quad or \quad}p\geq7, \\
   \widehat{A}_n, \widehat{D}_n, \widehat{E}_6, \widehat{E}_7 &\mbox{\quad if \quad }& p=5,\\
  \widehat{A}_n, \widehat{D}_n &\mbox{\quad if \quad }& p=3, \\
  \widehat{A}_n &\mbox{\quad if \quad }&p=2.
  \end{array}
 $$
\end{Theorem}

By construction, this bijection is compatible with the classical McKay correspondence via
the lifting results and the equivalences \eqref{eq: catequivalence},
\eqref{eq: rep thy intro}.

\begin{Example}
\label{sec: mup intro}
The linearly reductive group scheme corresponding to 
$\widehat{A}_n$ is $\bmu_{n+1}$.
This group scheme is reduced, that is, \'etale, if and only if
$p$ does not divide $n+1$.
In particular, it is crucial to allow non-reduced group schemes
in order to obtain a bijection as in characteristic zero.
\end{Example}

\subsection{Linearly reductive quotient singularities}
Consider $\GL_{2,k}$ with its usual linear action on 
$U=\Aff^2_k$ or $U=\widehat{\Aff}_k^2$.
If $G$ is a finite, linearly reductive, and \emph{very small}
(see Definition \ref{def: very small}) subgroup scheme of $\GL_{2,k}$, then the associated
quotient singularity $x\in X:=U/G$ is a 
two-dimensional \emph{linearly reductive quotient singularity} 
in the sense of \cite{LMM}.
By loc.cit., such a singularity determines 
$G$ together with its embedding $G\to\GL_{2,k}$ up to isomorphism and
conjugation, respectively.

In Section \ref{subsec: Ishii Ito Nakamura}, we will see 
that a minimal resolution of singularities of a two-dimensional linearly reductive
quotient singularity $x\in X=U/G$
is provided by the $G$-Hilbert scheme
\begin{equation}
\label{eq: intro IIN}
 \pi \,:\, \GHilb{G}(U) \,\to\, U/G,
\end{equation}
which generalises work of Ishii, Ito, and Nakamura 
\cite{Ishii,IN}.

\begin{Remark}
In dimension two, a linearly reductive quotient singularity
is the same as an F-regular singularity \cite{LMM}.
Thus, every two-dimensional F-regular singularity can be 
resolved by a suitable $G$-Hilbert scheme, see 
also Remark \ref{rem: f-regular}.
\end{Remark}

If moreover $G$ is a subgroup scheme of $\SL_{2,k}$, 
then $x\in X=U/G$ is called a \emph{Klein singularity}.
Klein singularities are rational double point singularities.
If $p=0$ or $p\geq7$, then conversely every rational double point 
is a Klein singularity by Hashimoto \cite{Hashimoto} and, independently, by 
\cite{LiedtkeSatriano}.
If $p\in\{2,3,5\}$, then not every rational double point 
is a Klein singularity.

\subsection{Canonical lifts and simultaneous resolutions}
Let $G$ be a very small, finite, and linearly reductive subgroup scheme of $\GL_{2,k}$ 
and let $x\in X=U/G$ be the associated linearly reductive quotient singularity.
Assume $p>0$ and let $W(k)$ be the ring of Witt vectors of $k$.
In Section \ref{subsec: canonical lift singularity}, 
we will establish the existence of a 
\emph{canonical lift} 
$$
 \mathcal{X}_{\mathrm{can}}\,\to\,\Spec W(k)
$$
of  $x\in X=U/G$.
Using $G$-Hilbert schemes in families, we 
will see in Section \ref{subsec: simultaneous resolution} 
that it admits a simultaneous and minimal resolution of singularities
$$
 \widetilde{\pi} \,:\, \mathcal{Y}\,\to\,\mathcal{X}_{\mathrm{can}}\,\to\,\Spec W(k).
$$
We will prove this resolution to be unique, see 
Theorem \ref{thm: IN}.

\subsection{McKay correspondence for Klein singularities}
Let $G$ be a finite and linearly reductive subgroup scheme of $\SL_{2,k}$ 
and let $x\in X=U/G$ be the associated Klein singularity.

As discussed in Section \ref{subsec: McKay graph intro}, we have
the McKay graph $\widehat{\Gamma}$ associated to 
$G$, its embedding into $\SL_{2,k}$, and the
set of isomorphism classes of simple representations of $G$.

Let $\pi:Y\to X$ be a minimal resolution of singularities.
Since $x\in X$ is a rational double point, the exceptional divisor $\mathrm{Exc}(\pi)$ 
of $\pi$  is a configuration of $\PP^1$'s, whose
dual intersection graph $\Gamma$ is a simply-laced Dynkin diagram.

\begin{Theorem}[Theorem \ref{thm: bijection simple and exceptional}]
\label{thm: bijection simple intro}
 Let $k$ be an algebraically closed field of characteristic $p\geq0$.
 Let $G$ be a finite and linearly reductive subgroup scheme of $\SL_{2,k}$
 and let $x\in X=U/G$ be the associated Klein singularity.
 Then, there exists a natural bijection of the graph $\Gamma$ 
 with the graph obtained from $\widehat{\Gamma}$ by removing the vertex
 corresponding to the trivial representation.
 \end{Theorem}
 
 Since every rational double point in characteristic $p\geq7$ 
 is a Klein singularity, we obtain the following.

\begin{Corollary}
There exists a linearly reductive McKay correspondence
for rational double point singularities in every characteristic $p\geq7$.
\end{Corollary}

To establish this theorem, we use the Ishii-Ito-Nakamura resolution
of singularities \eqref{eq: intro IIN}, as well as Hecke correspondences as in the work
of Ito and Nakamura \cite{IN} and Nakajima \cite{Nakajima, NakajimaLectures}.

\subsection{Generalisations and variants}
Let $G$ be a very small, finite,
and linearly reductive subgroup scheme of $\GL_{2,k}$ 
and let $x\in X=U/G$ be the associated two-dimensional linearly
reductive quotient singularity.
Let $\pi:Y\to X$ be the minimal resolution of singularities and 
let $\mathrm{Exc}(\pi)$ be the exceptional divisor of $\pi$.
\begin{enumerate}
\item In Theorem \ref{thm: exc pi}, 
we associate a representation of $G$ to each point of $\mathrm{Exc}(\pi)$.
This generalises results of Ishii and Nakamura \cite{IshiiCrelle, Ishii}.
\item In Theorem \ref{thm: reflexive}, we establish a bijection between 
the components of $\textrm{Exc}(\pi)$ and reflexive 
$\calO_X$-modules.
Probably, this result should be viewed as a theorem on two-dimensional F-regular
singularities, see Remark \ref{rem: f-regular Artin Verdier}.
It generalises work of Artin and Verdier \cite{AV}, 
Wunram \cite{Wunram}, and Ishii and Nakamura \cite{Ishii}.
\item In Theorem \ref{thm: derived}, we establish an equivalence of derived
categories of coherent sheaves $\calD(Y)$ and 
$\calD^G(U)$ ($G$-equivariant 
sheaves on $U$).
This generalises work of  Kapranov and Vasserot \cite{KV},
Bridgeland, King, and Reid \cite{Bridgeland},
and Ishii, Nakamura, and Ueda \cite{IshiiCrelle,Ishii, IshiiUeda}.
Of course, these articles themselves generalise work of Gonzalez-Sprinberg and Verdier 
\cite{GSV} from K-theory to derived categories of coherent sheaves.
\end{enumerate}

\subsection{Ito-Reid correspondence}
If $G$ is a finite subgroup of $\SL_2(\CC)$, then 
Ito and Reid \cite{ItoReid} found a natural bijection,
the \emph{Ito-Reid correspondence},
between the conjugacy classes of $G$ and the vertices 
of the McKay graph $\widehat{\Gamma}$.

\begin{Theorem}
\label{thm: ito reid intro}
 Let $k$ be an algebraically closed field of characteristic $p\geq0$.
 There exists an Ito-Reid correspondence for finite and linearly reductive
 subgroup schemes of $\SL_{2,k}$.
\end{Theorem}

This can be proven using lifting results and the Ito-Reid correspondence
over $\CC$, see Section \ref{subsec: Ito-Reid} for details.

The main difficulty is to define a notion of conjugacy class
for finite and linearly reductive group schemes that makes
such a correspondence work:
Let us recall the ring $K_k(G)$ from Section \ref{subsubsec: rep thy}.
If $G$ is a finite and linearly reductive group scheme over $k$,
then we define the set of \emph{conjugacy classes} of $G$
to be
$$
\Spec \left(\CC\otimes K_k(G)\right).
$$
At first sight, this  might look rather artificial.
One should think of it as defining
conjugacy classes to be ``dual'' to simple representations. 
Moreover, our definition is compatible with 
lifting over $W(k)$ and it induces a bijection of the conjugacy
classes of $G$ with the conjugacy classes of $G_{\rm abs}$.
We refer to Section \ref{subsec: conjugacy classes}
and Appendix \ref{subsec: second approach} for details.

\subsection{Conjugacy classes}
\label{intro: conjugacy classes}

In Theorem \ref{thm: ito reid intro}, we had to find a definition
of conjugacy classes that makes this theorem true.
This begs for the question whether there are other definitions 
or approaches, which is a question that is interesting in its own.

Let $G$ be a finite group scheme (not necessarily linearly reductive)
over an algebraically closed field $k$.
In Appendix \ref{app: conjugacy class}, we study the following
approaches to the notion of the set of conjugacy classes of $G$:

\begin{enumerate}
\item The set of conjugacy classes $G(k)/\sim$ of the group of $k$-rational points of $G$.
\item The spectrum of $F\otimes K_k(G)$, where $F$ is a field of characteristic zero
that contains ``sufficiently many'' roots of unity.
\item The scheme that represents the functor of conjugacy classes from schemes over $k$ to sets
defined by $S\mapsto G(S)/\sim$,
\item the isotypical component of the trivial representation of the 
adjoint representation of $G$.
\item The simple subrepresentations of the extended adjoint representations
$^{\mathrm{Ad}}A$ and $^{\mathrm{Ad}}(A^*)$, which are defined using the
quantum doubles of the Hopf algebra $A:=H^0(G,\OO_G)$
and its dual $A^*$.
\end{enumerate}
All approaches lead essentially to the ``same answer'' in characteristic
zero.
However, they usually lead to very different notions in positive characteristic.
On the other hand, all approaches have their merits and drawbacks.
For linearly reductive group schemes, the approach (2) leads to a definition that is
compatible with lifting and that leads to an Ito-Reid correspondence.

\subsection{Organisation of this article}

\begin{enumerate}
\item[-] In Section \ref{sec: linearly reductive}, we recall basic facts about finite and 
linearly reductive group schemes over algebraically closed fields.
\item[-] In Section \ref{sec: McKay}, we construct the McKay graph $\widehat{\Gamma}$.
The main result  is Theorem \ref{thm: main}, a McKay correspondence.
\item[-] In Section \ref{sec: lrq singularities}, we recall basic facts about linearly
reductive quotient singularities.
We establish an Ishii-Ito-Nakamura-type resolution of singularities,
the canonical lift of such a singularity, and a unique simultaneous resolution
of singularities of the canonical lift.
\item[-] In Section \ref{sec: Hecke}, we revisit the Ishii-Ito-Nakamura-type resolution $\pi$
of singularities of $x\in X=U/G$
and we introduce Hecke correspondences, which leads to 
Theorem \ref{thm: bijection simple and exceptional},
a bijection between simple and non-trivial representations of $G$ 
and  components of $\pi$.
We also study generalisations of this result to two-dimensional linearly reductive 
quotient singularities.
\item[-] In Section \ref{sec: ItoReid}, we establish an Ito-Reid correspondence
between conjugacy classes of $G$ 
and exceptional divisors in the minimal resolution of singularities of $x\in X=U/G$.
\item[-] In Section \ref{sec: outlook}, we study derived categories
of $G$-equivariant sheaves on $U$ and on the minimal 
resolution $\pi:Y\to X=U/G$.
\item[-] In Appendix \ref{sec: Hopf}, we recall results on finite group schemes
from the point of view of Hopf algebras.
We recall the adjoint representation, quantum doubles, and the
extended adjoint representation.
\item[-] In Appendix \ref{app: conjugacy class}, we study several approaches
toward the notion of a conjugacy class for finite group schemes.
\end{enumerate}

\begin{VoidRoman}[Acknowledgements]
 I thank Frank Himstedt, Martin Lorenz, Yuya Matsumoto, Frans Oort, 
 Matt Satriano, Takehiko Yasuda, and the referees 
 for discussions and comments.
\end{VoidRoman}

\section{Linearly reductive group schemes}
\label{sec: linearly reductive}

In this section, we recall a couple of general facts about finite and linearly
reductive group schemes over algebraically closed fields.
We discuss the close relationship between such a group scheme $G$ 
and a certain abstracted group $G_{\mathrm{abs}}$ associated to it.
For the relationship to Hopf algebras, we refer to Appendix \ref{sec: Hopf}.

\subsection{Group schemes}
Let $G$ be a finite group scheme over an algebraically closed 
field $k$ of characteristic $p\geq0$.
Since $k$ is perfect, there is a short exact sequence
of finite group schemes over $k$
\begin{equation}
\label{connected etale}
   1\,\to\,G^\circ\,\to\,G\,\to\, G^{\et} \,\to\,1,
\end{equation}
where $G^\circ$ is the connected component of the identity and where $G^{\et}$
is an \'etale group scheme over $k$.
The reduction $G_{\mathrm{red}}\to G$ provides a canonical splitting of \eqref{connected etale}
and we obtain a canonical semi-direct product decomposition $G\cong G^\circ\rtimes G^\et$.
Since $k$ is algebraically closed, $G^{\et}$ is the constant group scheme associated
to the finite group $G(k)=G^{\et}(k)$ of $k$-rational points.
Moreover, $G^\circ$ is an infinitesimal group scheme of length equal to 
some power of $p$.
In particular, if $p=0$ or if the length of $G$ is prime to $p$, 
then $G^\circ$ is trivial, and then, $G$ is \'etale. 

If $M$ is a finitely generated abelian group, then the group algebra
$k[M]$ carries a Hopf algebra structure and the associated commutative
group scheme is denoted $D(M):=\Spec k[M]$.
By definition, such group schemes are called \emph{diagonalisable}.
For example, we have 
$D(\C_n)\cong\bmu_n$, where $\C_n$ denotes the cyclic group of order $n$.
We have that $\bmu_n$ is \'etale over $k$ if and only if
$p\nmid n$.

A finite group scheme $G$ over $k$ is said to be \emph{linearly reductive} 
if every $k$-linear and finite-dimensional representation of $G$ is semi-simple.
If $p=0$, then all finite group schemes over $k$ are \'etale and
linearly reductive.
If $p>0$, then, by a theorem that is often attributed to
Nagata \cite[Theorem 2]{Nagata61}
(but see also \cite[Proposition 2.10]{AOV}, \cite{Chin}, 
and \cite[Section 2]{Hashimoto}), 
a finite group scheme over $k$ is
linearly reductive if and only if it is an extension of a finite and \'etale
group scheme, whose length is prime to $p$, 
by a diagonalisable group scheme. 

\subsection{Abstract groups and canonical lifts}
\label{subsec: abstract groups}
Let $G$ be a finite and linearly reductive group scheme
over an algebraically closed field $k$ of characteristic $p>0$.
Following \cite[Section 2]{LMM}, we study the finite group
$$
G_{\mathrm{abs}} \,:=\, 
\left(\underline{((G^\circ)^{D}(k))}_{\CC}\right)^{D}(\CC) \,\rtimes\, G^\et(k).
$$
Here, $-^D$ denotes $\underline{\mbox{Hom}}(-,\GG_m)$, the \emph{Cartier dual}, of
a commutative and finite group scheme.
If $G$ is \'etale, then $G_{\mathrm{abs}}=G(k)$ and if $G=\bmu_{p^n}$,
then $G_{\mathrm{abs}}=\C_{p^n}$.
In any case, the order of $G_{\mathrm{abs}}$ is equal to the length of $G$.

\begin{Definition}
The finite group $G_{\mathrm{abs}}$
is called the \emph{abstract finite group} associated to $G$.
\end{Definition}

\begin{Lemma} \label{lem: lrequivalence}
The functor
$$
G \,\mapsto\, G_{\mathrm{abs}}
$$
establishes an equivalence of categories between the category of finite and 
linearly reductive group schemes over $k$ 
and the category of finite groups with a normal and abelian $p$-Sylow subgroup.
\end{Lemma}

\begin{proof}
\cite[Lemma 2.1]{LMM}.
\end{proof}

\begin{Remark}
A finite group $G$ with a normal and abelian $p$-Sylow subgroup $P$ is the same
as a finite group with a unique and abelian $p$-Sylow subgroup.
In this case, the Schur-Zassenhaus theorem implies that
$G$ is isomorphic to a semi-direct product $P\rtimes G/P$.
\end{Remark}

Next, we study lifts to characteristic zero:
let $W(k)$ be the ring of Witt vectors of $k$, let $K$ be its field
of fractions, and let $\overline{K}$ be an algebraic closure
of $K$.
By \cite[Proposition 2.4]{LMM},
there exist lifts of $G$ as finite and flat
group scheme over $W(k)$.
More precisely, $G^\circ$ and $G^\et$ even lift uniquely to $W(k)$,
but their extension class usually does not, see also
\cite[Example 2.6]{LMM}.
However, there is a unique lift $\calG_{\mathrm{can}}$
of $G$ to $W(k)$ that is characterised by being 
a semi-direct product of the unique lift of $G^\circ$ with the unique lift of 
$G^{\et}$.

\begin{Definition}\label{def: canonical lift}
The lift $\mathcal{G}_{\mathrm{can}}\to\Spec W(k)$ is called
the \emph{canonical lift} of $G$.
We set $G_{\mathrm{can}}:=\mathcal{G}_K\to \Spec K$
and also call it canonical lift.
\end{Definition}

Every other lift of $G$ to some extension of $R\supseteq W(k)$ differs from
$\calG_{\mathrm{can},R}$ by a twist and thus, there is only one \emph{geometric lift}
of $G$ to $\overline{K}$ up to isomorphism - namely $G_{\mathrm{can},\overline{K}}$,
see \cite[Section 2.2]{LMM}.

Since $\overline{K}$ is algebraically closed and
of characteristic zero, $G_{\mathrm{can},\overline{K}}$ is the constant
group scheme associated to the finite group
$G_{\mathrm{can}}(\overline{K})$.
In fact, we have $G^\circ\cong\prod_i \bmu_{p^{n_i}}$ for some $n_i$'s and if we set
$N:=\max\{n_i\}_i$, fix a primitive $p^N$.th root of unity $\zeta_{p^N}$ and set 
$K_N:=K(\zeta_{p^N})$, 
then $G_{\mathrm{can},K_N}$ is the constant group scheme associated 
to the finite group $G_{\mathrm{can}}(K_N)$ and we have
$G_{\mathrm{can}}(K_N)=G_{\mathrm{can}}(\overline{K})$.

\begin{Lemma}
 There exist isomorphisms of finite groups
 $$
   G_{\mathrm{abs}} \,\cong\, G_{\mathrm{can}}(K_N) \,=\,G_{\mathrm{can}}(\overline{K}).
 $$
In particular, there exist isomorphisms
 $$
  G_{\mathrm{can},K_N} \,\cong\,
 \left(\underline{G_{\mathrm{abs}}}\right)_{K_N}
 \mbox{ \quad and \quad }
 G_{\mathrm{can},\overline{K}} \,\cong\,
 \left(\underline{G_{\mathrm{abs}}}\right)_{\overline{K}}
 $$
 of finite group schemes over $K_N$ and $\overline{K}$, respectively.\qed
\end{Lemma}

\subsection{Representation theory}
\label{subsec: representation comparison}
Let $k$ be an algebraically closed field of characteristic $p>0$,
and let $W(k)$, $K$, and $\overline{K}$ be as in the previous section.
The equivalence of Lemma  \ref{lem: lrequivalence}
can be extended to representations and K-theory.
If $G$ is a finite group or a finite group scheme over $k$,
we denote by ${\rm Rep}_k(G)$ the category of its
$k$-linear and finite-dimensional representations.
Let us first recall \cite[Corollary 2.11]{LMM}.

\begin{Proposition} \label{prop: liftingrepresentation}
 Let $G$ be a finite and linearly reductive group scheme
 over $k$.
 Then, there exist canonical equivalences of categories
 $$
   \mathrm{Rep}_{k}(G) 
   \,\to\, \mathrm{Rep}_{K_N}(G_{\mathrm{can},K_N})
   \,\to\, \mathrm{Rep}_{\overline{K}}(G_{\mathrm{can},\overline{K}})
   \,\to\, \mathrm{Rep}_\CC(G_{\mathrm{abs}}),
 $$
 which are compatible with degrees, direct sums, tensor products, duals, 
 and simplicity.\qed
\end{Proposition}

\begin{Remark}
 In particular, this allows us to define \emph{characters} or even a
 \emph{character table} of $G$ via 
 Proposition \ref{prop: liftingrepresentation} and $G_{\mathrm{abs}}$.
\end{Remark}

Let $K_k(G)$ be the $K$-group associated to ${\rm Rep}_k(G)$.
In fact, $K_k(G)$ has a natural structure of a commutative ring with one, 
where the sum (resp. product) structure comes from direct sums (resp. tensor products)
of representations.
A straight forward application of Proposition \ref{prop: liftingrepresentation}
is the following.

\begin{Corollary}
 \label{cor: Ktheory}
 There exist isomorphisms of rings
 $$
  K_k(G) \,\to\, 
  K_{K_N}(G_{\mathrm{can},K_N}) \,\to\,
  K_{\overline{K}}(G_{\mathrm{can},\overline{K}}) \,\to\,
  K_{\CC}(G_{\mathrm{abs}}).\qed
 $$
\end{Corollary}

\section{McKay graph and McKay correspondence}
\label{sec: McKay}

In this section, we introduce the \emph{McKay graph}
associated to a finite and linearly reductive subgroup scheme $G$
over an algebraically closed field $k$ of characteristic $p\geq0$
and a representation $\rho:G\to\GL_{n,k}$.
This induces a bijection between certain affine Dynkin diagrams
and finite and linearly reductive subgroup schemes of $\SL_{2,k}$.
As an application, we establish a \emph{linearly reductive McKay correspondence}.

\subsection{McKay graph}
\label{subsec: McKay graph}
Let $G$ be a finite and linearly reductive group scheme
over an algebraically closed field $k$ of characteristic $p\geq0$.
Let $\{\rho_i\}_i$ be the finite set of isomorphism classes of $k$-linear 
and simple representations of $G$.
Following tradition, we assume that $\rho_0$ is the trivial representation.
We fix a representation $\rho:G\to\GL(V)$.
If $G$ is a subgroup scheme of $\SL_{n,k}$ or $\GL_{n,k}$, 
then $\rho$ is usually the linear representation corresponding to the embedding
of $G$ into this linear algebraic group.
By assumption, $\mathrm{Rep}_k(G)$ is semi-simple.
Therefore, there exist unique integers $a_{ij}\in\ZZ_{\geq0}$ for each $i$,
such that we have isomorphisms of $k$-linear representations
$$
  \rho\otimes\rho_i \,\cong\, \bigoplus_{j}\, \rho_j^{\oplus a_{ij}}.
$$
Associated to this data, we define the \emph{McKay graph}, 
denoted $\Gamma(G,\{\rho_i\},\rho)$:
\begin{enumerate}
\item[-] The vertices are the $\{\rho_{i}\}_i$.
(Some sources exclude the
the trivial representation $\rho_0$.)
\item[-] There are $a_{ij}$ edges from the vertex corresponding to
$\rho_i$ to the vertex corresponding to $\rho_j$.
\end{enumerate}
We now establish a couple of elementary properties of
this graph, which are well-known in the classical case
and which immediately carry over to the linearly reductive
situation.
We leave the proof of the first lemma to the reader.

\begin{Lemma}
 We have
 $$
  a_{ij} \,=\, \dim_k\, \Hom(\rho_i,\rho_j\otimes\rho).
 $$
 In particular, if $\rho$ is self-dual, that is, $\rho\cong\rho^\vee$,
 then $a_{ij}=a_{ji}$ for all $i,j$.
 In this case, we can consider  $\Gamma(G,\{\rho_i\},\rho)$ 
 as an undirected graph.\qed
\end{Lemma}

\begin{Lemma}
\label{lem: selfdual}
Let $\rho:G\to\SL_{2,k}$ be a homomorphism of group schemes
over $k$, considered as a
2-dimensional representation.
Then, $\rho$ is self-dual.
\end{Lemma}

\begin{proof}
Being a 2-dimensional representation, 
$\rho^\vee$ is isomorphic to $\rho\otimes\det(\rho)$ and
the lemma follows.
\end{proof}

\begin{Lemma}
\label{lem: faithful}
Let $\rho:G\to\GL_{n,k}$ be a faithful representation.
Then, every irreducible representation of $G$
occurs as subrepresentation of $\rho^{\otimes m}$ for
some suitable $m$.
\end{Lemma}

\begin{proof}
This is well-known for finite groups.
Using Proposition \ref{prop: liftingrepresentation}, it 
carries over to finite and linearly reductive group schemes.
\end{proof}

\begin{Corollary}
If $\rho$ is a faithful representation, then the graph
$\Gamma(G,\{\rho_i\},\rho)$  is connected.
\end{Corollary}

\begin{proof}
It suffices to note that  the number of paths in $\Gamma$
of length $m$ that connect the vertices corresponding to $\rho_i$ and $\rho_j$
is equal to the multiplicity of $\rho_i$ in $\rho_j\otimes\rho^{\otimes m}$,
see, for example, the proof of \cite[Theorem 8.13]{Kirillov}.
\end{proof}

Let $k$ be an algebraically closed field of characteristic $p>0$,
and let $W(k)$, $K$, and $\overline{K}$ be as in Section \ref{subsec: abstract groups}.
There, we also discussed the
canonical lift $G_{\mathrm{can}}$ of $G$ over $K$ and we saw that there exists
an isomorphism of finite groups
$G_{\mathrm{abs}}\cong G_{\mathrm{can}}(\overline{K})$.
Proposition \ref{prop: liftingrepresentation} implies the following result.

\begin{Proposition}
\label{prop: McKay Quiver lift}
Let $G$ be a finite and linearly reductive group scheme over $k$.
Let $\{\rho_i\}_i$ be the set of isomorphism classes of 
simple representations of $G$.
Let $\rho$ be a finite-dimensional representation of $G$.
\begin{enumerate}
    \item 
    Proposition \ref{prop: liftingrepresentation} yields
    sets of representations $\{\rho_{\mathrm{can},i}\}_i$ and $\{\rho_{\mathrm{abs},i}\}_i$
    of $G_{\mathrm{can}}$ and $G_{\mathrm{abs}}$, respectively, which are the 
    sets of isomorphism classes of simple representations of $G_{\mathrm{can}}$
    and $G_{\mathrm{abs}}$, respectively.
    \item 
    Proposition \ref{prop: liftingrepresentation} yields
    representations $\rho_{\mathrm{can}}$ and $\rho_{\mathrm{abs}}$
    of $G_{\mathrm{can}}$ and $G_{\mathrm{abs}}$, respectively.
 \end{enumerate}
This data leads to a bijection of the three McKay graphs
     $$
     \Gamma(G,\{\rho_i\}_i,\rho),\quad
     \Gamma(G_{\mathrm{can}},\{\rho_{\mathrm{can},i}\}_i,\rho_{\mathrm{can}}),\quad
     \Gamma(G_{\mathrm{abs}},\{\rho_{\mathrm{abs},i}\}_i,\rho_{\mathrm{abs}}).\qed
     $$
\end{Proposition}

\subsection{McKay correspondence}
We now run through the classical McKay correspondence \cite{McKay} 
in our setting, where we follow \cite[Section 8.3]{Kirillov}.
Let $k$ be an algebraically closed field of characteristic $p\geq0$
and let $G$ be a finite and linearly reductive subgroup scheme
of $\SL_{2,k}$.
\begin{enumerate}
    \item The K-group $K_k(G)$ carries a symmetric bilinear form 
    $$
     ([V],[W])_0 \,:=\,  \dim_k\,\mathrm{Hom}_G(V,W).
    $$
    \item We consider the closed embedding $G\to\SL_{2,k}$ as a 
    2-dimensional 
    representation 
    $\rho:G\to\SL_{2,k}\to\GL_{2,k}$
    and define an operator 
    $$
     A\,:\,K_k(G)\,\to\, K_k(G),\mbox{\qquad} [V]\mapsto [V]\otimes\rho\,.
    $$
    Since $\rho$ is self-dual by Lemma \ref{lem: selfdual}, it follows
    that $A$ is symmetric with respect to $(-,-)_0$, see \cite[Lemma 8.12]{Kirillov}.
    \item Using $\rho$, we define a symmetric bilinear form on 
    $K_k(G)\otimes_\ZZ\RR$ by
    $$
     ([V],[W]) \,:=\, \left([V],\, (2-A)[W]\right)_0,
    $$
    which is positive semi-definite.
    The class $\delta\in K_k(G)$ of the regular representation 
    of $G$ generates the radical of $(-,-)$.
    \item Let $\{\rho_i\}$ be the set of isomorphism classes of simple
    representations of $G$.
    \begin{enumerate}
    \item
    The McKay graph $\widehat{\Gamma}=\Gamma(G,\{\rho_i\}_i,\rho)$
    is an affine Dynkin diagram
    of type $\widehat{A}_n$, $\widehat{D}_n$ with $n\geq4$, 
    $\widehat{E}_6$, $\widehat{E}_7$
    or $\widehat{E}_8$.
    \item
    After removing the vertex corresponding to the trivial representation
    $\rho_0$ from $\widehat{\Gamma}$, 
    we obtain a finite Dynkin diagram 
    of type $A_n$, $D_n$ with $n\geq4$, $E_6$, $E_7$
    or $E_8$, respectively.
    \end{enumerate}
\end{enumerate}
If $G$ is a finite subgroup of $\SL_2(\CC)$, then these statements are part
of the classical McKay correspondence, see \cite{McKay} or 
\cite[Section 8.3]{Kirillov}.
In our setting of finite and linearly reductive subgroup schemes of $\SL_{2,k}$,
the above claims immediately follow from the classical McKay correspondence 
together with the lifting results Proposition \ref{prop: liftingrepresentation}
and Proposition \ref{prop: McKay Quiver lift}.
We then obtain the following analog of 
McKay's theorem \cite{McKay}.

\begin{Theorem}
\label{thm: main}
 Let $k$ be an algebraically closed field of characteristic $p\geq0$.
 There exists a bijection between
 finite, non-trivial, and linearly reductive subgroup schemes of
 $\SL_{2,k}$ up to conjugation and affine Dynkin diagrams
 of type
 $$
  \begin{array}{lcl}
   \widehat{A}_n, \widehat{D}_n, \widehat{E}_6, \widehat{E}_7, \widehat{E}_8
       &\mbox{\quad if \quad }&    p=0\mbox{\quad or \quad}p\geq7, \\
   \widehat{A}_n, \widehat{D}_n, \widehat{E}_6, \widehat{E}_7 &\mbox{\quad if \quad }& p=5,\\
  \widehat{A}_n, \widehat{D}_n &\mbox{\quad if \quad }& p=3, \mbox{ and} \\
  \widehat{A}_n &\mbox{\quad if \quad }&p=2.
  \end{array}
 $$
\end{Theorem}

\begin{proof}
If $p=0$, then this is part of the classical McKay correspondence.
If $p>0$, then this follows from the linearly reductive McKay correspondence just discussed together 
with Hashimoto's classification \cite[Theorem 3.8]{Hashimoto}
of finite and linearly reductive subgroup schemes of $\SL_{2,k}$.
\end{proof}

\begin{Remark}
The linearly reductive group schemes corresponding to $\widehat{E}_6$, $\widehat{E}_7$, and
$\widehat{E}_8$ are \'etale and correspond to finite groups of order prime to $p$.
The linearly reductive group scheme corresponding to $\widehat{A}_n$ (resp.
$\widehat{D}_n$) is \'etale if and only if $p\nmid (n+1)$ (resp. $p\nmid (n-2)$).
We refer to \cite[Proposition 4.2]{LiedtkeSatriano} for details.
Thus, even if $p$ is sufficiently large, then one does not obtain a bijection in 
Theorem \ref{thm: main} with finite groups of order to prime to $p$ only.
\end{Remark}

\begin{Remark}
Steinberg \cite{Steinberg} established many properties of the McKay graph
and the McKay correspondence for finite subgroups of $\SU_2(\CC)$
\emph{without} using classification lists.
It seems plausible to obtain a proof of Theorem \ref{thm: main} along these 
lines without using lifting results or classification lists.
\end{Remark}

\section{Linearly reductive quotient singularities}
\label{sec: lrq singularities}

In this section, we recall \emph{linearly reductive quotient singularities}
in the sense of \cite{LMM} and some general results, including
the \emph{canonical lift} of such a singularity over the ring of Witt vectors.
In dimension two, we establish a minimal resolution of singularities using $G$-Hilbert schemes
as in the work of Ishii, Ito, and Nakamura \cite{IshiiCrelle, Ishii, IN}.
We also show that the canonical lift admits a unique minimal and simultaneous
resolution of singularities.
Finally, we discuss \emph{Klein singularities} and their relation to
rational double point singularities.

\subsection{Quotient singularities}
Let $k$ be an algebraically closed field of characteristic $p\geq0$.
If $G$ is a finite and linearly reductive group scheme over $k$,
if $V$ is a finite-dimensional $k$-vector space,
and if $\rho:G\to\GL(V)$ is a linear representation, 
then we define the \emph{$\lambda$-invariant} of $\rho$ 
as in \cite[Definition 2.7]{LMM} to be
$$
\lambda(\rho) \,:=\, \max_{\{{\mathrm{id}}\}\neq\bmu_n\subseteq G} \dim V^{\bmu_n},
$$
where $V^{\bmu_n}$ denotes the subspace of $\bmu_n$-invariants.
As explained in \cite[Remark 2.8]{LMM}, the representation
$\rho$ is faithful if and only if $\lambda(\rho)\neq\dim V$.
Moreover, $\rho$ contains no pseudo-reflections if and only
if $\lambda(\rho)\leq\dim V-2$ and in this case,
the representation $\rho$ is said to be \emph{small}.

\begin{Definition}
\label{def: very small}
 The representation $\rho$ is called 
 \emph{very small} if $\lambda(\rho)=0$.
\end{Definition}

Note that in dimension two the notions of small and very small coincide.
We refer to \cite[Section 3]{LMM} for the classification of  
finite and linearly reductive group schemes that admit a very small
representation.
By \cite[Proposition 6.5]{LMM}, $\lambda(\rho)$ is equal to the dimension
of the non-free locus of the induced $G$-action on the spectrum of the
formal power series ring $k[[V^\vee]]=(k[V^\vee])^\wedge$.
Following \cite[Definition 6.4]{LMM} and using the linearisation result
\cite[Proposition 6.3]{LMM}, we define:

\begin{Definition}
 \label{def: lrq singularity}
 A \emph{linearly reductive quotient singularity} over $k$
 is an isolated singularity that is analytically isomorphic to
 $\Spec k[[V^\vee]]^G$, where $G$ is a finite and linearly reductive group scheme over
 $k$ and where $\rho:G\to\GL(V)$ is a very small representation.
\end{Definition}

\begin{Remark}
We set $U:=\Spec k[V^\vee]$ 
or $U:=\Spec k[[V^\vee]]$ and simply write
$x\in X=U/G$ with the assumptions of
Definition \ref{def: lrq singularity} implicitly understood.
\end{Remark}

Properties of these singularities have been studied  in \cite{LMM}:
for their local \'etale fundamental groups, class groups, and
F-signatures, we refer to \cite[Section 7]{LMM}.

By \cite[Theorem 8.1]{LMM}, a linearly reductive quotient singularity 
determines the finite and linearly reductive group scheme $G$
together with the very small representation $\rho:G\to\GL(V)$
uniquely up to isomorphism and conjugacy, respectively.
It thus makes sense to refer to $\rho(G)\subset\GL(V)$
as the \emph{finite and linearly reductive subgroup scheme of $\GL(V)$ 
associated to the linearly reductive quotient singularity}
$x\in X=U/G$.
In particular, the classification of linearly reductive quotient singularities 
in dimension $d$ is ``the same'' as the classification of very small, finite, 
and linearly reductive subgroup schemes of $\GL_{d,k}$ up to conjugacy.
We refer to \cite[Section 3]{LMM} for details and this classification.

\subsection{Minimal resolution of singularities}
If $x\in X$ is a normal surface singularity, then it 
admits a unique minimal resolution of singularities
$$
  \pi\,:\, Y\,\to\,X.
$$
If $x\in X$ is moreover a rational singularity, then the 
exceptional locus of $\pi$ is a union
of $\PP^1$'s, whose dual intersection graph $\Gamma$ contains no cycles.
The graph $\Gamma$ is called the \emph{type} of $x\in X$.
If the type determines the singularity up to analytic isomorphism,
then the singularity is said to be \emph{taut}.

Over $\CC$, taut singularities have been classified by 
Laufer \cite{Laufer}.
For example, two-dimensional finite quotient singularities 
over $\CC$ are taut.
Since two-dimensional linearly reductive quotient singularities over algebraically
closed fields of positive characteristic are 
F-regular (see \cite[Proposition 7.1]{LMM}), they are taut by Tanaka's
theorem \cite{Tanaka}.

\begin{Remark}
\label{rem: type quotient singularity}
Very small, finite, and linearly reductive subgroup schemes of $\GL_{2,k}$
have been classified in \cite[Theorem 3.4]{LMM}, extending Brieskorn's classification 
\cite[Satz 2.9]{Brieskorn} of small subgroups of $\GL_2(\CC)$.
The types of the associated quotient singularities in terms of this classification are 
given by \cite[Satz 2.11]{Brieskorn}.

Finite and linearly reductive subgroup schemes of $\SL_{2,k}$ are automatically
very small and they
have been classified by Hashimoto \cite[Theorem 3.8]{Hashimoto}, extending 
Klein's classification \cite{Klein} of finite subgroups of $\SL_2(\CC)$
and we refer to \cite{Hashimoto, LiedtkeSatriano} for the types
of the associated quotient singularities
(see also Theorem \ref{thm: classification Klein} below).
\end{Remark}

\subsection{The Ishii-Ito-Nakamura resolution}
\label{subsec: Ishii Ito Nakamura}
In \cite{IN}, Ito and Nakamura showed that if 
$G$ is a finite subgroup of $\SL_2(\CC)$, then a minimal
resolution of singularities of $\CC^2/G$ is provided
by the $G$-Hilbert scheme $\GHilb{G}(\CC^2)$.
Ishii \cite{IshiiCrelle} extended this to quotient singularities
$\CC^2/G$ for $G$ a finite and very small subgroup of $\GL_2(\CC)$
and Ishii and Nakamura \cite{Ishii} extended this to quotient
singularities $U/G$ for $G$ a finite and very small subgroup
of $\GL_2(k)$ of order prime to $p$.

Let $k$ be an algebraically closed field of characteristic $p\geq0$.
 Let $G$ be a very small, finite, and linearly reductive subgroup scheme of 
 $\GL_{2,k}$.
 Set $U:=\Aff^2_k$ or $\widehat{\Aff}_k^2$ 
 and let $x\in X:=U/G$ be the associated linearly reductive
 quotient singularity.
 By \cite{Blume}, there exists a $G$-Hilbert scheme
$\GHilb{G}(U)$ over $k$ that parametrises zero-dimensional
$G$-invariant subschemes $Z\subset U$ 
(so-called \emph{clusters}), such that the $G$-representation on 
$H^0(Z,\OO_Z)$ is the regular representation.
Taking a cluster $Z$ to its $G$-orbit
(see, for example, \cite[Remark 3.3]{Blume}) 
induces a morphism
$$
\pi\,:\,Y\,:=\,\GHilb{G}(U) \,\to\, U/G\,=\,X
$$
over $k$.

\begin{Theorem}
\label{thm: IshiiItoNakamura}
 The morphism $\pi$ is a minimal resolution of singularities.
\end{Theorem}

\begin{proof}
If $p=0$, then this is shown in \cite{Ishii}, extending \cite{IN}.
If $p>0$ and $G$ is of length prime to $p$, then this is shown in \cite{Ishii}.
However, this proof also works if $G$ is linearly reductive and $p$ divides the length
of $G$.
\end{proof}

\begin{Remark}
\label{rem: f-regular}
By \cite[Theorem 11.5]{LMM}, every two-dimensional F-regular singularity
$x\in X$ is a linearly reductive quotient singularity, say $x\in X=U/G$ for
some finite and linearly reductive subgroup scheme $G$ of $\GL_{2,k}$.
By \cite[Theorem 8.1]{LMM}, $G$ and its embedding into $\GL_{2,k}$ are unique
up to isomorphism and conjugacy, respectively.
By Theorem \ref{thm: IshiiItoNakamura}, $\GHilb{G}(U)\to U/G=X\ni x$
is the minimal resolution of singularities.
In particular, every two-dimensional F-regular singularity can be resolved
by a suitable $G$-Hilbert scheme.
\end{Remark}

\subsection{Canonical lifts}
\label{subsec: canonical lift singularity}
Let $k$ be an algebraically closed field of characteristic $p>0$,
let $W(k)$ be the ring of Witt vectors, let $K$ be its field of fractions, and let
$\overline{K}$ be an algebraic closure.
Let $x\in X=U/G$ be a $d$-dimensional 
linearly reductive singularity, where $U=\Aff^d_k$ or $U=\widehat{\Aff}_k^d$.
Next, we set $\calU:=\Aff^d_{W(k)}$ or $\widehat{\Aff}^d_{W(k)}$, respectively.
In Section \ref{subsec: abstract groups}, we recalled the canonical lift 
$\calG_{\mathrm{can}}$ of $G$ over $W(k)$.
By \cite[Proposition 2.9]{LMM}, there exists a unique lift of the $G$-action
from $U$ to $\calU$.
From this, we obtain a flat family 
\begin{equation}
\label{eq: lrq lift}
  \mathcal{X}_{\mathrm{can}}\,:=\,\calU/\calG_{\mathrm{can}} \,\to\, \Spec W(k)\,.
\end{equation}
of linearly reductive quotient singularities over $W(k)$, whose special fibre
over $k$ is $x\in X=U/G$.

\begin{Definition}
 The family \eqref{eq: lrq lift} is called the \emph{canonical lift} of the linearly reductive 
 quotient singularity $x\in X=U/G$.
 \end{Definition}
  
By the Lefschetz principle, the geometric generic fibre of $\mathcal{X}_{\mathrm{can}}$
can be identified with a finite quotient singularity of the form $\CC^d/G_{\mathrm{abs}}$, 
where $G_{\mathrm{abs}}$ is the abstract group associated to $G$
and the embedding of $G_{\mathrm{abs}}\to\GL_d(\CC)$
corresponds to the embedding $G\to\GL_{d,k}$ provided by
Proposition \ref{prop: liftingrepresentation}.
The canonical lift is unique in the following sense:

\begin{Proposition}
  We keep the notations and assumptions.
   Let $W(k)\subseteq R$ be a finite extension of complete DVRs and let $\mathcal{X}\to\Spec R$ 
   be a lift of $x\in X$ that is of the form $\calV/\calG\to\Spec R$ for some flat lift
   $\calG$ of $G$ to $R$ and $\calV\cong\calU\times_RS$.
   Then, there exists a finite extension $R\subseteq S$ of complete DVRs, such that
   \begin{enumerate}
      \item There exists an isomorphism 
      $$
        \calG_{\mathrm{can}} \times_{W(k)}S \,\cong\, \calG\times_RS
      $$
     of group schemes over $S$.
 \item The very small representation $\rho:G\to\GL_{d,k}$ corresponding to the $G$-action on $U$
 lifts uniquely to $\calG_{\mathrm{can}}$ and $\calG$, respectively, and they become conjugate
 over $S$. 
 \item  There is an isomorphism 
  $$
     \mathcal{X}_{\mathrm{can}}\times_{W(k)}S \,\cong\, \mathcal{X}\times_RS
   $$
   of deformations of $x\in X$ over $S$.
 \end{enumerate}
\end{Proposition}

\begin{proof}
Since $\calG$ and $\calG_{\mathrm{can},R}$ are lifts of $G$
to $R$, they become isomorphic after passing to a finite extension $R\subseteq S$,
see also the discussion in Section \ref{subsec: abstract groups}.
There, we also saw that the linear representation $\rho:G\to\GL_{d,k}$ 
lifts uniquely to $\calG$ and $\calG_{\mathrm{can}}$, respectively,
and that they become conjugate over $S$.
From this, we deduce an isomorphism
$\mathcal{X}_{\mathrm{can}}\times_{W(k)}S\cong\mathcal{X}\times_RS$.
\end{proof}

\begin{Remark}
If $d\geq3$, then \cite[Corollary 10.10]{LMM} shows that
a $d$-dimensional linearly reductive quotient singularity $x\in X$ admits
precisely one lift over $W(k)$, namely the canonical lift.
If $d=2$, then linearly reductive quotient singularities usually have positive dimensional deformation
spaces and admit many non-isomorphic
lifts to $W(k)$, see \cite[Section 12]{LMM}.
\end{Remark}

\subsection{Simultaneous resolution of singularities}
\label{subsec: simultaneous resolution}
If $\mathcal{X}\to S$ is a deformation of a rational double point singularity $x\in X$,
then it admits a simultaneous resolution of singularities, but usually only after some
finite base-change $S'\to S$, see \cite{ArtinBrieskorn}.
In the case where $\mathcal{X}\to S$ is a family of rational surface singularities, then
such a finite $S'\to S$ exists if $S$ maps to the so-called \emph{Artin component}
inside the versal deformation space of $x\in X$.
Moreover, due to the existence of flops, 
these simultaneous resolutions (if they exist) are not unique in general.

In the special case where $x\in X$ is a two-dimensional linearly reductive
quotient singularity over some algebraically closed field $k$ of
characteristic $p>0$ and where $\mathcal{X}_{\mathrm{can}}\to\Spec W(k)$ is the canonical
lift of $x\in X$, we will now show that there exists a simultaneous and minimal resolution of singularities over $W(k)$
and that it is unique.
This simultaneous resolution can be most elegantly constructed using 
the Ishii-Ito-Nakamura resolution from Section \ref{subsec: Ishii Ito Nakamura}
in families.

Let $k$ be an algebraically closed field of characteristic $p>0$.
Let $G$ be a finite and linearly reductive group scheme over $k$ 
and let $\rho:G\to\GL_{2,k}$ be a very small representation.
Let $\calG_{\mathrm{can}}\to\Spec W(k)$
be the canonical lift of $G$ and let 
$\widetilde{\rho}:\calG_{\mathrm{can}}\to\GL_{2,W(k)}$ be the lift of $\rho$ to $W(k)$.
Let $\calU:=\Aff^2_{W(k)}$ and let 
$$
 \mathcal{X}_{\mathrm{can}} \,:=\, \calU/\calG_{\mathrm{can}} \,\to\, \Spec W(k)
$$
be the canonical lift of $x\in X$.
By \cite{Blume}, there exists a $\calG_{\mathrm{can}}$-Hilbert scheme
$$
\GHilb{\calG_{\mathrm{can}}}(\calU) \,\to\,\Spec W(k)
$$
that parametrises $\calG_{\mathrm{can}}$-invariant subschemes $Z\subset\calU$ 
that are finite and flat over $W(k)$ (so-called $\calG_{\mathrm{can}}-$\emph{clusters})
and such that the $\calG_{\mathrm{can}}$-representation on 
$H^0(Z,\OO_Z)$ is the regular representation.
Taking such a cluster $Z$ to its $\calG_{\mathrm{can}}$-orbit
(see, for example, \cite[Remark 3.3]{Blume}) 
induces a morphism
\begin{equation}
\label{eq: IIN in families}
\widetilde{\pi}\,:\,
\mathcal{Y}\,:=\,\GHilb{\calG_{\mathrm{can}}}(\calU) \,\to\, \calU/\calG_{\mathrm{can}}\,=\,\mathcal{X}_{\mathrm{can}}
\end{equation}
over $W(k)$.

\begin{Theorem}
\label{thm: IN}
 Keeping the assumptions and notations
 $$
  \widetilde{\pi} \,:\, \mathcal{Y} \,\to\, \mathcal{X}\,\to\,\Spec W(k)
 $$
 is a simultaneous minimal resolution of singularities of the canonical lift
 $\mathcal{X}\to\Spec W(k)$ of the linearly reductive quotient singularity
 $x\in X=U/G$.
 \begin{enumerate}
 \item The simultaneous resolution $\widetilde{\pi}$ is unique up to isomorphism.
 \item The exceptional locus of $\widetilde{\pi}$ is a union of $\PP^1_{W(k)}$'s
 meeting transversally
 and we denote by $\Gamma$ its dual intersection graph.
 \item The special fibre and the generic
 fibre are linearly reductive quotient singularities of type $\Gamma$
 and $\widetilde{\pi}$ identifies the components of the exceptional loci
 of $\widetilde{\pi}_k$ and $\widetilde{\pi}_K$.
 \end{enumerate}
\end{Theorem}

\begin{proof}
We recall that we defined $G_{\mathrm{can}}:=\calG_{\mathrm{can},K}$, which
will simplify the notation in the following.
The generic and special fibre of \eqref{eq: IIN in families}
over $K$ and $k$ are isomorphic to
$$
\begin{array}{cccccc}
&\GHilb{G_{\mathrm{can}}}(\calU_K) &\to& \calU_K/G_{\mathrm{can}} &\to&\Spec K\\
\mbox{and \quad}&
\GHilb{G}(U) &\to& U/G &\to&\Spec k,
\end{array}
$$
respectively, by \cite[Remark 3.1]{Blume}.
Now, 
$\GHilb{G_{\mathrm{can},\overline{K}}}(\calU_{\overline{K}}) \to \calU_{\overline{K}}/G_{\mathrm{can},\overline{K}}$
and
$\GHilb{G}(U) \to U/G$,
are minimal resolutions of singularities by Theorem \ref{thm: IshiiItoNakamura}.
Thus, 
$\GHilb{G_{\mathrm{can}}}(\calU_K) \to \calU_K/G_{\mathrm{can}}$
is a resolution of singularities, which is minimal since it is minimal over $\overline{K}$.
Thus, $\mathcal{Y}\to\mathcal{X}\to\Spec W(k)$ is a simultaneous minimal resolution 
of singularities.

The exceptional fibres $\widetilde{\pi}_K$ and $\widetilde{\pi}_k$ of $\widetilde{\pi}$ over $K$ and $k$ 
are unions of $\PP^1$'s that intersect transversally.
The types, that is dual resolution graphs, associated to $\widetilde{\pi}_K$ and $\widetilde{\pi}_k$
are the same (see also Remark \ref{rem: type quotient singularity}) 
and we denote this graph by $\Gamma$.
In particular, the exceptional fibres of $\widetilde{\pi}_K$ and $\widetilde{\pi}_k$ have
the same numbers of irreducible components.
In particular, the specialisation maps of N\'eron-Severi lattices
$$
\NS(\mathcal{Y}_K) \,\leftarrow\, \NS(\mathcal{Y}) \,\to\, \NS(\mathcal{Y}_k)
$$
are isometries of lattices.
These identify the components of the exceptional loci of $\widetilde{\pi}_K$
and $\widetilde{\pi}_k$.
Moreover, given such a component of the exceptional locus of $\widetilde{\pi}_k$,
it is isomorphic to $\PP^1_k$ and it extends uniquely to a $\PP^1_{W(k)}$ 
in the exceptional locus of $\widetilde{\pi}$.
This establishes claims (2) and (3).

It remains to prove claim (1):
Let $\mathcal{Y}'\to\mathcal{X}\to \Spec W(k)$ be a simultaneous resolution
of singularities that coincides with the minimal resolution on special and generic fibres,
respectively.
Let $\alpha:\mathcal{Y}'_K\,\to\,\mathcal{Y}_K$ be an isomorphism over $\mathcal{X}_K$
and choose a relatively (to $\mathcal{X}$)
ample invertible sheaf $\mathcal{L}$ on $\mathcal{Y}_K$.
Then, $\alpha^*\mathcal{L}$ is relatively ample on $\mathcal{Y}_K$.
Since the types of the singularities of $\mathcal{X}_K$ and $\mathcal{X}_k$
are the same, the fibres of
$\mathcal{Y}'_K\to\mathcal{X}_K$ and $\mathcal{Y}'_k\to\mathcal{Y}_k$
contain the same number of exceptional divisors.
Thus, the specialisation
map of N\'eron-Severi lattices $\NS(\mathcal{Y}'_K)\to\NS(\mathcal{Y}'_k)$
is an injective map between two lattices of the same rank and the same discriminant,
whence an isometry.
Similarly, the specialisation
map of N\'eron-Severi lattices $\NS(\mathcal{Y}_K)\to\NS(\mathcal{Y}_k)$
is an isometry.
In particular, we can identify the components of the exceptional loci 
of $\mathcal{Y}_K$ and $\mathcal{Y}_k$.
(If the types of $\mathcal{X}_k$ and $\mathcal{X}_K$ differ, then this specialisation map is 
usually only injective, there are usually more $(-2)$-curves in the special fibre and this is also the
place where the difficulties with flops begin.)
In particular, $\mathcal{L}$ and $\alpha^*\mathcal{L}$ extend to relative ample
invertible sheaves on $\mathcal{Y}$ and $\mathcal{Y}'$, respectively.
By \cite[Theorem 5.14]{Kovacs}, the isomorphism $\alpha$
extends to an isomorphism $\mathcal{Y}'\to\mathcal{Y}$ over $\mathcal{X}$ and 
the claimed uniqueness follows.
\end{proof}

\subsection{Rational double point singularities and Klein singularities}
We specialise Definition \ref{def: lrq singularity} to the following case.

\begin{Definition}
\label{def: Klein singularity}
 A \emph{Klein singularity} is a linearly reductive quotient singularity
 as in Definition \ref{def: lrq singularity} with $\dim V=2$
 and $\det\rho=1$, that is, $\rho$ is a homomorphism of $G$ 
 to $\SL_{2,k}$.
\end{Definition}

In particular, a Klein singularity is a two-dimensional
and linearly reductive quotient singularity, it is a rational surface singularity,
and since $\det(\rho)=1$, it is Gorenstein.
Quite generally, rational and Gorenstein surface singularities
are precisely the \emph{rational double point singularities} 
\cite{ArtinRational}.
We refer to \cite{Durfee} or \cite{Seminaire} for more background
on surface singularities and rational double point singularities.

If $x\in X$ is a rational double point, then its type $\Gamma$ is a 
simply-laced finite Dynkin graph.
In characteristic zero, every rational double point singularity is taut.
In positive characteristic, a Klein singularity is F-regular and thus, taut.
In the case of rational double points,
this also follows from Artin's explicit classification \cite{ArtinRDP}.
On the other hand, rational double point singularities that are not F-regular 
need not be taut.
We have the following relation between rational double point singularities
and Klein singularities in positive characteristic, see 
\cite[Theorem 11.2]{LMM}.

\begin{Theorem}
\label{thm: classification Klein}
Let $k$ be an algebraically closed field of characteristic $p\geq0$.
Let $x\in X$ be a normal surface singularity over $k$.
\begin{enumerate}
\item If $p=0$ or $p\geq7$, then $x\in X$ is a Klein singularity if and only if $x\in X$ is a rational double point singularity.
\item If $p>0$, then $x\in X$ is a Klein singularity if and only if $x\in X$ is a rational double point singularity and F-regular. 
The finite Dynkin graphs of these singularities are of type
$$
  \begin{array}{lcl}
   A_n, D_n, E_6, E_7, E_8
       &\mbox{\quad if \quad }&    p\geq7, \\
   A_n, D_n, E_6, E_7 &\mbox{\quad if \quad }& p=5,\\
   A_n, D_n &\mbox{\quad if \quad }& p=3, \\
   A_n &\mbox{\quad if \quad }&p=2.
  \end{array}
 $$
\end{enumerate}
\end{Theorem}

A Klein singularity $x\in X=U/G$ is a linearly reductive quotient singularity and thus,
determines $G$ and $\rho:G\to\SL_{2,k}$ up to isomorphism and conjugacy, respectively.
Thus, the classification of Klein singularities boils down to the classification
of finite and linearly reductive subgroup schemes of $\SL_{2,k}$ and we refer to
Remark \ref{rem: type quotient singularity} for this classification.

\begin{Remark}
Theorem \ref{thm: main} and Theorem \ref{thm: classification Klein} rely on the
classification of finite and linearly reductive subgroup schemes
of $\SL_{2,k}$.
It is therefore no surprise that the classification lists coincide.
\end{Remark}

\section{Hecke correspondences}
\label{sec: Hecke}

Let $x\in X=U/G$ be a Klein singularity over an algebraically closed field $k$ 
of characteristic $p>0$.
In Section \ref{subsec: abstract groups}, we discussed the
canonical lift $\mathcal{X}\to \Spec W(k)$ of $x\in X$ and 
in Section \ref{subsec: simultaneous resolution} we established
a simultaneous minimal resolution of singularities of the canonical lift.
More precisely, this simultaneous resolution was constructed
using $G$-Hilbert schemes extending previous work of Ishii, Ito, and Nakamura
\cite{IshiiCrelle, Ishii, IN}.

In this section, we refine this resolution of singularities as in
the work of Ito and Nakamura \cite{IN} and Nakajima \cite{Nakajima, NakajimaLectures}:
we eventually obtain a bijection between the components
of the minimal resolution of singularities of the Klein singularity $x\in X=U/G$
and the simple and non-trivial representations of $G$ using special
 \emph{Hecke correspondences}.

\subsection{The Ito-Nakamura resolution revisited}
We  first slightly extend the setup of Section \ref{subsec: Ishii Ito Nakamura}.
Let $k$ be an algebraically closed field of characteristic $p>0$ and
let $x\in X=U/G$ be a Klein singularity as in Definition \ref{def: Klein singularity},
that is, we have $U=\Aff^2_k$ and $G$ a very small, finite, and linearly reductive
subgroup scheme of $\SL_{2,k}$.

Let $\{\rho_i\}_{i\in I}$ be the set of isomorphism classes of simple representations of $G$.
Given a finite-dimensional representation $\rho$ of $G$, there exist
non-negative integers $\nu_i\in\ZZ_+$, such that $\rho$ is isomorphic
to $\bigoplus_i\rho_i^{\oplus \nu_i}$ and we combine these into a multi-index
$\nu=(\{\nu_i\})\in\ZZ_{\geq0}^I$.
We set $\dim(\nu):=\sum_i\nu_i\dim\rho_i$, which is the dimension
of the representation associated to $\nu$.

For any integer $n\geq1$, the $G$-action on $U$ induces
a $G$-action on the Hilbert scheme $\Hilb^n(U)$, which 
parametrises zero-dimensional subschemes of length $n$ of $U$. 
We consider the fixed point scheme
$$
\Hilb^{n,G}(U) \,:=\,
\left( \Hilb^n(U) \right)^G,
$$
that is, the largest subscheme of $\Hilb^n(U)$
on which $G$ acts trivially.
It parametrises $G$-invariant and zero-dimensional subschemes of
length $n$ of $U$.
Given $\nu=(\{\nu_i\})_i\in\ZZ_{\geq0}^I$ with $\dim(\nu)=n$, we define
$$
   H^\nu\,=\,\left\{
   Z\,\in\,\Hilb^{n,G}(U)\,|\,
   H^0(Z,\calO_Z)\,\cong\,
   \bigoplus \rho_i^{\oplus \nu_i}\mbox{ as $G$-representation}
   \right\},
$$
which defines a subscheme of $\Hilb^{n,G}(U)$.
Adapting \cite[Lemma 12.4]{Kirillov}, which follows \cite[Section 9]{IN},
to our situation, we have the following.

\begin{Lemma}
We keep the assumptions and notations and let 
$n\geq1$ be an integer.
\begin{enumerate}
    \item The scheme $\Hilb^{n,G}(U)$ is smooth over $k$.
    \item We have a decomposition
    $$
    \Hilb^{n,G}(U) \,=\, \bigsqcup_{\nu}\, H^\nu,
    $$
    where the disjoint union runs over all multi-indices $\nu\in\ZZ_{\geq0}^I$ 
   of dimension $n$.
    Each $H^\nu$ is a smooth subscheme of $\Hilb^{n,G}(U)$.
\end{enumerate}
\end{Lemma}

\begin{proof}
Since $U$ is two-dimensional, $\Hilb^n(U)$ is smooth over $k$ by
\cite[Theorem 2.4]{Fogarty}.
To show Claim (1), it remains to show that the fixed point scheme for the
$G$-action is also smooth,
which follows from Lemma \ref{lem: fixed scheme} below.
Then, Claim (2) is obvious.
\end{proof}

\begin{Lemma}
\label{lem: fixed scheme}
Let $k$ be an algebraically closed field,
let $X$ be a scheme that is smooth over $k$,
let $G$ be a finite and linearly reductive group scheme over $k$,
and assume that $G$ acts on $X$.
Then, the fixed point scheme $X^G\subseteq X$ is smooth over $k$.
\end{Lemma}

\begin{proof}
Let $x\in X^G$.
Passing to the completion of the local ring $\calO_{X,x}$,
using that $X$ is smooth over $k$, and passing to
coordinates such that the $G$-action is linear
(this is always possible by \cite[proof of Corollary 1.8]{Satriano}),
we may assume that 
$$
\widehat{\calO}_{X,x} \,\cong\, k[[u_1,...,u_d]]
$$
and that the $G$-action is linear.
In this description it is easy to see that the $G$-invariant subscheme
of $\Spec\widehat{\calO}_{X,x}$ is smooth
(see also the proof of \cite[Lemma 9.1]{IN}), which implies that
$X^G$ is smooth near $x$.
\end{proof}

An important special case is the \emph{regular representation} of $G$,
where we have $\nu_i=\dim\rho_i$ for all $i$ and in  this case, 
we will write $\delta$ for the corresponding multi-index.
The dimension of $\delta$ is equal to the length of $G$.
In this case, we have 
$$
   \pi\,:\, H^\delta \,=\, \GHilb{G}(U) \,\to\, U/G \,=\,X,
$$
which we already studied in Section \ref{subsec: Ishii Ito Nakamura}.
There, we saw that $\pi$ is a minimal resolution
of singularities of the Klein singularity $x\in X=U/G$.

\subsection{Hecke correspondences}
We keep the assumptions and notations of the previous section.
We let $n$ be the length of $G$ and let $\delta$ be the multi-index
corresponding to the regular representation.
For $i\in I$, we set $\alpha_i:=(0,...,0,1,0...)$ with the non-zero entry in 
the $i$.th position.
We define
$$
 B_i \,:=\, \left\{ J_1 \,\in\, H^{\delta-\alpha_i},\quad J_2\,\in\, H^\delta\quad |\quad J_2\subseteq J_1 \right\}
 \,\subseteq\, H^{\delta-\alpha_i}\times H^\delta,
$$
which is a \emph{Hecke correspondence},
see \cite{Nakajima}, \cite[Section 6.1]{NakajimaLectures} or \cite[Section 12.4]{Kirillov}.
We let
$$
   E_i \,:=\, \mathrm{pr}_2(B_i)\,\subseteq\, H^\delta \,=\, \GHilb{G}(U)
$$
be the image under the projection $\mathrm{pr}_2$ onto the second factor.

In characteristic zero, the following result is due to Nakajima \cite{Nakajima}
and independently to Ito and Nakamura \cite{IN}, see also 
\cite[Section 12.4]{Kirillov}.

\begin{Theorem}
\label{thm: bijection simple and exceptional}
 Let $k$ be an algebraically closed field of characteristic $p\geq0$.
 Let $G$ be a finite and linearly reductive subgroup scheme of $\SL_{2,k}$,
 let $x\in X:=U/G$ be the associated Klein singularity, and 
 let $\pi:H^\delta\to X$ be the Ito-Nakamura resolution.
 \begin{enumerate}
 \item The assignment $i\mapsto E_i$ defines a bijection
 between the set $\{\rho_i\}_i$ of isomorphism classes of 
 non-trivial simple representations  of $G$ 
 and the set of irreducible components of the exceptional divisor $\mathrm{Exc}(\pi)$ 
 of $\pi$.
 \item If $i\neq j$, then  
  $$
    B_i\cdot B_j \,=\, A_{ij}
  $$
  where the numbering of the $\{\rho_i\}_i$ is as in
  \cite[Section 10]{IN}.
  \end{enumerate}
We thus obtain a bijection between the dual resolution graph of $\pi$
and the graph obtained by removing the vertex $\rho_0$ from the
 McKay graph of $G$ with respect to the embedding $G\to\SL_{2,k}\to\GL_{2,k}$.
\end{Theorem}

\begin{proof}
The case-by-case analysis of \cite[Section 12]{IN} carries directly over
from finite subgroups of $\SL_2(\CC)$ to finite and linearly reductive subgroup
schemes of $\SL_{2,k}$.
\end{proof} 

Here is a second proof of this theorem that uses our lifting results:

\begin{proof}
If $p=0$, then the Lefschetz principle allows to reduce to $k=\CC$,
where the theorem is due to Ito, Nakajima, and Nakamura \cite{IN, Nakajima}.

We now assume $p>0$.
Let  $\mathcal{X}_{\mathrm{can}}=\calU/\calG_{\mathrm{can}}\to\Spec W(k)$
be the canonical lift of $x\in X=U/G$
(see Section \ref{subsec: canonical lift singularity})
and let $\widetilde{\pi}:\mathcal{Y}\to\mathcal{X}_{\mathrm{can}}\to\Spec W(k)$ 
be the simultaneous resolution of singularities
(see Section \ref{subsec: simultaneous resolution}).
Moreover, we have the abstract group $G_{\mathrm{abs}}$
and the canonical lift $G_{\mathrm{can}}$
associated to $G$ (see Section \ref{subsec: abstract groups})
and we have a bijection of simple representations of 
$G_{\mathrm{abs}}$, $G_{\mathrm{can}}$, and $G$
by Proposition \ref{prop: liftingrepresentation}.

The generic fibre $\mathcal{X}_K$ is isomorphic
to $\calU_K/G_{\mathrm{can}}$ and using the Lefschetz principle,
the geometric generic fibre can be identified with $\CC^2/G_{\mathrm{abs}}$.
Over $\CC$, we have Theorem \ref{thm: bijection simple and exceptional}
for $G_{\mathrm{abs}}$, $\CC^2/G_{\mathrm{abs}}$, and its minimal resolution of singularities
by \cite{IN, Nakajima}.

By the Lefschetz principle, we have it 
for $G_{\mathrm{can}}$, 
$\mathcal{X}_K=\calU_{K}/G_{\mathrm{can}}$ 
and $\widetilde{\pi}_K$.
Using the comparison results 
Proposition \ref{prop: liftingrepresentation} and Theorem \ref{thm: IN},
we obtain it for
$\mathcal{X}_k=\calU_k/\calG_{\mathrm{can},k}$ 
and $\widetilde{\pi}_k$, that is,
for $X=U/G$, $G$ and $\pi$.
\end{proof}

\subsection{Local McKay correspondence}
\label{subsec: local McKay}
We now extend work of Ishii and Nakamura \cite{IshiiCrelle, Ishii} to our linearly 
reductive setting.
Let $G$ be a very small, finite, and linearly reductive subgroup scheme of $\GL_{2,k}$
and let $x\in X=U/G$ be the associated two-dimensional linearly reductive quotient
singularity.
Until the end of this section, we do not require it to be a Klein singularity.
Note that in dimension two, linearly reductive quotient singularities
are the same as F-regular singularities, see Remark \ref{rem: f-regular}.

Let $\{\rho_i\}_i$ be the set of simple representations of $G$.
Let $\rho_0$ be the trivial representation and we choose our numbering of the $\rho_i$'s 
to be the one of \cite[Theorem 3.6]{Ishii}.
Let $\pi:Y\to X$ be its minimal resolution
of singularities and let $\{E_i\}$ be the irreducible components of the
exceptional divisor $\textrm{Exc}(\pi)$ of $\pi$.
Then, there is exists a connection between the $\{\rho_i\}$ and the
exceptional divisors of $\pi$ as follows - this is yet another version of the 
McKay correspondence.

\begin{Theorem}
\label{thm: exc pi}
 Keeping assumptions and notations, let $\idealm\subset\OO_U$ be the maximal
 ideal corresponding to the origin, let $y\in Y$ be a closed point, let $Z_y$ be
 the $G$-invariant cluster of $U$ corresponding to $y$, and let $\calI_{Z_y}\subset\OO_U$
 be its ideal sheaf.
 Then, the $G$-representation on $\calI_{Z_y}/\idealm\calI_{Z_y}$ is given by
 $$
    \left\{
    \begin{array}{ll}
    \rho_i\,\oplus\,\rho_0 &\mbox{ \quad if }y\in E_i\,\backslash\,\bigcup_{j\neq i}E_j\\
    \rho_i\,\oplus\,\rho_j\,\oplus\,\rho_0  & \mbox{ \quad if }y\in E_i\cap E_j\mbox{ \quad with \quad }i\neq j.
    \end{array}
    \right. 
 $$
\end{Theorem}

\begin{proof}
For $k=\CC$, this is \cite[Theorem 7.1]{IshiiCrelle}.
For $k$ algebraically closed of arbitrary characteristic and $G$ a very small and finite
subgroup of $\GL_{2}(k)$ of order prime to $p$, this is \cite[Theorem 3.6]{Ishii}.
However, these proofs also work if $G$ is a very small and finite and linearly reductive
subgroup scheme of $\GL_{2,k}$.
\end{proof}

\begin{Remark}
In \cite{Ishii}, the $G$-quiver structure of $\GHilb{G}(U)$
was studied in the case where $G$ is a very small and finite subgroup of
$\GL_{2}(k)$. 
We leave the extension of these results to the linearly reductive case to the reader.
\end{Remark}

\subsection{Reflexive sheaves on the minimal resolution}
We end this section by shortly digressing on work of Wunram \cite{Wunram},
Ishii and Nakamura \cite{Ishii}, which generalises work of
Artin and Verdier \cite{AV}.
We keep the assumptions and notations from Section \ref{subsec: local McKay}.
Let $F\subset Y$ be the \emph{fundamental divisor} of $\pi$,
see \cite{ArtinRational}.

\begin{Theorem}
\label{thm: reflexive}
Keeping assumptions and notations, there exists a bijection between
\begin{enumerate}
\item the set of irreducible components $\{E_i\}$ of $\mathrm{Exc}(\pi)$ and
\item the set of non-trivial indecomposable full $\calO_Y$-modules $\{\calM_i\}$,
special in the sense that $H^1(Y,\calM_i^{\vee})=0$
\end{enumerate}
This correspondence $\calM_i\mapsto E_i$ is defined by
$$
 c_1(\calM_i)\cdot E_j \,=\, \delta_{i,j}\,.
$$
The rank of $\calM_i$ is equal to $c_1(\calM_i)\cdot F$, the 
multiplicity of $E_i$ in $F$.\qed
\end{Theorem}

\begin{proof}
For $k=\CC$, this is \cite{Wunram}.
For $k$ algebraically closed of arbitrary characteristic and $G$ a very small and finite
subgroup of $\GL_{2,k}$ of order prime to $p$, this is \cite[Theorem 2.8]{Ishii}.
However, these proofs also work if $G$ is a very small, finite, and linearly reductive
subgroup scheme of $\GL_{2,k}$.
\end{proof}

Here, a \emph{full} $\OO_Y$-module is as defined in \cite[Definition 2.4]{Ishii}.
By  \cite[Corollary 2.5]{Ishii}, the assignment $\calM\mapsto\pi_*\calM$
sets up a bijection between the set of (indecomposable) 
full $\calO_Y$-modules with the set of (indecomposable) 
reflexive $\OO_X$-modules.
In particular, the above theorem yields a bijection between the set of irreducible
components of $\mathrm{Exc}(\pi)$ and  the set of non-trivial, indecomposable, and
reflexive $\OO_X$-modules.

\begin{Remark}
\label{rem: f-regular Artin Verdier}
The group scheme $G$ plays no r\^{o}le in this theorem and the discussion thereafter.
Probably, these results should be viewed as results 
on two-dimensional F-regular 
singularities (see Remark \ref{rem: f-regular}).
\end{Remark}

\section{Conjugacy classes and Ito-Reid correspondence}
\label{sec: ItoReid}

For a finite group, the number of isomorphism classes of
complex simple representations is equal to the number 
of conjugacy classes.
Thus, one can \emph{choose} a bijection between these two sets.
In particular, for a finite subgroup of $\SL_2(\CC)$
one can \emph{choose} a bijection between the vertices of its 
McKay graph and the conjugacy classes of $G$.
In \cite{ItoReid}, Ito and Reid gave a \emph{canonical}
bijection between these two sets.
In this section, we extend this \emph{Ito-Reid correspondence}
to finite and linearly reductive subgroup schemes of
$\SL_{2,k}$.
The main difficulty is to define a suitable notion of \emph{conjugacy classes}
for finite and linearly reductive group schemes.

\subsection{Conjugacy classes}
\label{subsec: conjugacy classes}

Given a finite group $G_{\mathrm{abs}}$, a
representation $\rho:G_{\mathrm{abs}}\to\GL_n(\CC)$,
and an element $g\in G_{\mathrm{abs}}$, we have the 
trace $\mathrm{tr}(\rho(g))\in\CC$.
This pairing $(\rho,g)\mapsto \mathrm{tr}(\rho(g))$
induces a non-degenerate pairing between
isomorphism classes of 
simple representations of $G_{\mathrm{abs}}$
and conjugacy classes.
Thus, conjugacy classes can be thought of as being
``dual'' to semi-simple representations, which
can be used to give an unusual definition of
conjugacy classes.
This definition generalises to 
finite and linearly reductive group schemes.
More precisely, we make the following definition,
which is inspired by a result
of Serre \cite[Section 11.4]{Serre} and which we discuss
in detail in Appendix \ref{subsec: second approach}.

\begin{Definition}
  \label{def: conjugacy class}
 Let $G$ be a finite and linearly reductive group scheme over an
 algebraically closed field $k$ of characteristic $p\geq0$.
 Then, the \emph{set of conjugacy classes} is defined to be
 $\Spec \CC\otimes K_k(G)$.
\end{Definition} 

We discuss several approaches to conjugacy classes for finite group schemes 
in Appendix \ref{app: conjugacy class} -
there, we hope to convince the reader that Definition \ref{def: conjugacy class}
is the best for the purposes of this article.
For example, by Proposition \ref{prop: conjugacy lift and bijection}
it is compatible with canonical lifts and there is a natural
bijection with the set of conjugacy classes of the abstract group 
$G_{\mathrm{abs}}$ associated to $G$.
Concerning the choice of field $\CC$ in Definition \ref{def: conjugacy class},
we refer to Remarks \ref{rem: conjugacy classes}.

\subsection{An explicit bijection}
For a finite and linearly reductive subgroup scheme $G$ of $\SL_{2,k}$,
we have an associated embedding of the finite group
$G_{\rm abs}$ into $\SL_2(\CC)$.
Let $\{\rho_i\}_i$ be the set of isomorphism classes of
semi-simple representations of $G$ and 
let $\{\rho_{\mathrm{abs},i}\}_i$ be the corresponding
set for $G_{\mathrm{abs}}$ obtained by Proposition \ref{prop: liftingrepresentation}.
The associated McKay graphs
$\Gamma(G,\rho,\{\rho_i\}_i)$ and 
$\Gamma(G_{\mathrm{abs}},\rho_{\mathrm{abs}}, \{\rho_{\mathrm{abs},i}\}_i)$
coincide by Proposition \ref{prop: McKay Quiver lift}
and we denote both by $\widehat{\Gamma}$.
Next, the group $G_{\mathrm{abs}}$ admits a presentation of the form
$$
 G_{\rm abs} \,=\, \left\langle A,B,C \,|\, A^r=B^s=C^t=ABC \right\rangle
$$
for suitable non-negative integers $r,s,t$.
The non-trivial conjugacy classes of $G_{\rm abs}$ can be uniquely represented
by the following elements
$$
ABC,\quad A^i \mbox{ with }1\leq i\leq r-1,
 \quad B^i \mbox{ with }1\leq i\leq s-1,
 \quad C^i \mbox{ with }1\leq i\leq t-1.
$$
This allows to give an explicit bijection between the 
conjugacy classes of $G_{\mathrm{abs}}$ and the vertices of 
$\widehat{\Gamma}$.
We refer to \cite[Section 2]{Kirillov} for details.
Using Proposition \ref{prop: conjugacy lift and bijection}, we obtain 
an explicit bijection between the conjugacy classes of $G$
and the vertices of $\widehat{\Gamma}$.

\subsection{The Ito-Reid correspondence}
\label{subsec: Ito-Reid}
If $G$ is a finite subgroup of $\SL_2(\CC)$, then Ito
and Reid \cite{ItoReid} (see also \cite{Reid}) 
constructed a \emph{canonical} bijection.
Let us run through \cite[Section 2]{Reid} and show that this
can be carried over to our linearly reductive setting:
let $k$ be an algebraically closed field of characteristic $p>0$
and let $G$ be a finite and linearly reductive subgroup scheme
of $\SL_{2,k}$.
\begin{enumerate}
\item Associated to $G$, we have the abstract group $G_{\mathrm{abs}}$.
By Proposition \ref{prop: liftingrepresentation}, an embedding
$G\to\SL_{n,k}$ yields an embedding $G_{\mathrm{abs}}\to\SL_n(\CC)$.
By Proposition \ref{prop: conjugacy lift and bijection}, we can identify
conjugacy classes of $G$ (in the sense of Definition \ref{def: conjugacy class}) 
and conjugacy classes of $G_{\mathrm{abs}}$, which allows
us to define the \emph{age} of a conjugacy class of $G$ 
via the corresponding notion for $G_{\mathrm{abs}}$ as, for example,
in \cite[Section 2]{Reid}.
A conjugacy class of age 1 is called \emph{junior} and if $n=2$,
then all conjugacy classes are junior.
\end{enumerate}

\begin{Remark}
Using the ``toric mechanism'' mentioned in \cite[Section 2]{Reid},
one can define the age directly and without referring to lifts,
but we will not pursue this here.
\end{Remark}

\begin{enumerate}
\setcounter{enumi}{1}
\item Let $x\in X:=U/G$ be the associated Klein singularity. 
We have the canonical lift $\mathcal{X}_{\mathrm{can}}\to \Spec W(k)$ and the 
simultaneous resolution of singularities 
$\widetilde{\pi}:\mathcal{Y}\to\mathcal{X}_{\mathrm{can}}\to\Spec W(k)$ 
by Theorem \ref{thm: IN}.
Passing to geometric generic fibres and using the Lefschetz principle,
we obtain the minimal resolution of singularities of $\CC^2/G_{\mathrm{abs}}$.
The special fibre of $\widetilde{\pi}$, the geometric generic fibre
of $\widetilde{\pi}$ and the minimal resolution of $\CC^2/G_{\mathrm{abs}}$
are \emph{crepant} in the sense of \cite{YPG} and we can identify the
exceptional divisors of these three resolutions with each other,
see Theorem \ref{thm: IN}.
This way, we obtain an Ito-Reid correspondence between
junior conjugacy classes of $G$ and crepant divisors of the resolution
\cite[Theorem 2.1]{Reid}.
\end{enumerate}

\begin{Remark}
It seems reasonable that one can extend this correspondence
to finite and linearly reductive subgroup schemes of $\SL_{n,k}$ with $n\geq3$,
but we will not pursue this here.
\end{Remark}

\section{Derived categories}
\label{sec: outlook}

Let $G$ be a very small, finite, and linearly reductive subgroup scheme of $\GL_{2,k}$,
let $x\in X:=U/G$ be the associated  linearly reductive
quotient singularity, and let $\pi:Y\to X$ be its minimal resolution of singularities.
Gonzalez-Sprinberg and Verdier \cite{GSV} gave an interpretation of the McKay correspondence
as an isomorphism between the K-groups $K^G(U)$ and $K(Y)$.
Kapranov and Vasserot \cite{KV} and Bridgeland, King, and Reid \cite{Bridgeland}  generalised this
to an equivalence of derived categories $\calD^G(U)$ and $\calD(Y)$. 
In this section, we extend this to our setting, following
Ishii, Ito, Nakamura, and Ueda \cite{IshiiCrelle, Ishii, IshiiUeda}.

We have a commutative diagram
 $$
 \xymatrix{
  U\times_k Y \ar[r]^{\pi_U}\ar[d]^{\pi_Y} &U\ar[d]^\varpi\\
   Y \ar[r]^\pi & X.
  }
 $$
By Theorem \ref{thm: IshiiItoNakamura}, the minimal resolution $\pi$ can be constructed by
the Ishii-Ito-Nakamura resolution
$$
 \GHilb{G}(U) \,\to\, U/G
$$
We let $\calZ$ be the universal cluster over $\GHilb{G}(U)$,
we identify $\GHilb{G}(U)$ with $Y$
and then, we have a commutative diagram
  $$
 \xymatrix{
  \calZ \ar[r]^{q}\ar[d]^{p} &U\ar[d]^\varpi\\
   Y \ar[r]^\pi & X.
  }
 $$
Let $\calD(Y)$ be the derived category of coherent
sheaves on $Y$.
Let $\calD^G(U)$ be the derived category of
$G$-equivariant coherent sheaves on $U$.
Following \cite{IshiiCrelle} and \cite{Ishii}, we define two functors
$$
\begin{array}{ccccc}
 \Psi &:& \calD^G(U) &\to& \calD(Y)\\
 \Phi &:& \calD(Y) &\to& \calD^G(U)
\end{array}
$$
by
$$
\Psi(-) \,:=\, [p_*\, \mathbf{L}q^*\, (-)]^G
$$
and
$$
\Phi(-) \,:=\, \mathbf{R}\pi_{U,*} \left(
\OO_{\calZ}^\vee \,\otimes^{\mathbf{L}}\,\pi_Y^*(-\otimes\rho_0)\,\otimes^{\mathbf{L}}\, \pi_U^*K_{U}\right) [2],
$$
where $\OO_{\calZ}^\vee:=\mathbf{R}Hom(\OO_{\calZ},\OO_{Y\times U})$ denotes the dual of $\OO_{\calZ}$,
where $-\otimes\rho_0:\calD(Y)\to\calD^G(Y)$ denotes the functor that attaches the trivial $G$-action, and where
$K_U$ denotes the canonical sheaf of $U$.
We refer to \cite[Section 3.1]{Ishii} for details, conventions, and notations.

\begin{Theorem}
\label{thm: derived}
 Keeping assumptions and notations, 
 $\Phi$ is fully faithful and $\Psi$ is a left adjoint of $\Phi$.
\end{Theorem}

\begin{proof}
For $G$ is a very small subgroup of $\GL_2(\CC)$, this is \cite[Section 6]{IshiiCrelle} and 
\cite[Proposition 1.1 and Lemma 2.9]{IshiiUeda}.
For $k$ algebraically closed of arbitrary characteristic and $G$ a very small and finite
subgroup of $\GL_{2}(k)$ of order prime to $p$, this is \cite[Theorem 3.2]{Ishii}.
However, these proofs also work if $G$ is a very small, finite, and linearly reductive
subgroup scheme of $\GL_{2,k}$.
\end{proof}

\appendix
\section{Hopf algebras}
\label{sec: Hopf}

In the first section of the appendix, we study finite group schemes from the point of view of 
finite-dimensional Hopf algebras.
We also recall the quantum double of a Hopf algebra,  as well as the adjoint and
the extended adjoint representation.
Many results of this section should be well-known to the experts,
but are somewhat scattered over the literature - especially, since
many sources (sometimes implicitly) assume characteristic zero 
or work even over the complex numbers.

\subsection{Generalities}
\label{subsec: Hopf}
If $G$ is a finite group scheme over a field $k$, then the $k$-algebra
$A:=H^0(G,\calO_G)$ is commutative and finite-dimensional as $k$-vector space.
The multiplication $m:G\times G\to G$, the inverse $i:G\to G$, and 
the neutral element $e:\Spec k\to G$
induce $k$-algebra homomorphisms
$m^*:A\to A\otimes_k A$, $i^*:A\to A$, and $e^*:A\to k$,
which turn $A$ into a co-algebra over $k$ with co-multiplication
$m^*$ and antipode $S:=i^*$, and thus,
into a \emph{Hopf algebra} over $k$.
Since $i$ is the inverse of $G$, the antipode $S$ satisfies
$S^2={\rm id}_A$, that is, $A$ is an \emph{involutive} Hopf algebra.

Conversely, if $A$ is a finite-dimensional and commutative Hopf algebra
over $k$, then it is involutive and $\Spec A$ is a finite group scheme over $k$.
Moreover, $A$ is a co-commutative Hopf algebra if and only if $G$ is a 
commutative group scheme.

\begin{Example}
\label{ex: group-like}
 Let $G_{\rm abs}$ be a finite group.
 The group algebra $k[G_{\rm abs}]$ becomes a Hopf algebra 
 by defining the co-multiplication to be $\Delta(g)=g\otimes g$ and the antipode 
 to be $S(g)=g^{-1}$.
 \begin{enumerate}
 \item Clearly, $k[G_{\rm abs}]$ is an involutive and co-commutative Hopf algebra.
 Moreover, $k[G_{\rm abs}]$ is commutative if and only if $G_{\rm abs}$ is commutative.
\item There exist examples of non-isomorphic finite groups $H_{\rm abs}$ and $G_{\rm abs}$,
whose group rings $k[H_{\rm abs}]$ and $k[G_{\rm abs}]$
are isomorphic as $k$-algebras.

However, they are not isomorphic as Hopf algebras:
Recall that an element $x\in B$ in a Hopf algebra $B$
is called \emph{group-like} if $\Delta(x)=x\otimes x$.
The set of group-like elements of $B$ form a group.
If $B=k[G_{\rm abs}]$, then
the group of group-like elements of $B$ is isomorphic to
$G_{\rm abs}$ and thus, recovers the group.
 \item There exists an isomorphism of finite group schemes over $k$
$$
  G\,\cong\,\Spec k[G_{\rm abs}]^*, 
$$
where $k[G_{\rm abs}]^*$ denotes the dual Hopf algebra and where $G$ denotes
the constant group scheme over $k$ associated to $G_{\rm abs}$.
\end{enumerate}
\end{Example}

\subsection{Simplicity}
A Hopf algebra $A$ is \emph{co-semi-simple} 
if its dual Hopf algebra $A^*$ is semi-simple and $A$ is \emph{bi-semi-simple}
if both $A$ and $A^*$ are semi-simple.

Let $G$ be a finite group scheme over $k$
and let $A:=H^0(G,\calO_G)$ be the associated Hopf algebra.
To give a finite-dimensional
and $k$-linear representation $\rho:G\to{\rm GL}(V)$ is the
same as to give an $A$-co-module $V\to V\otimes_k A$,
see \cite[Section 3.2]{Waterhouse} for details.
Using this equivalence, we see that $G$ is linearly reductive if and only if 
$A$ is co-semi-simple.
For the classification of linearly reductive group schemes
(see Section \ref{sec: linearly reductive})
in the language of Hopf algebras we refer to \cite{Chin}.

The following equivalences are probably well-known to the experts.

\begin{Proposition}
\label{prop: hopf simple}
 Let $G$ be a finite group scheme
 over an algebraically closed field $k$ of characteristic $p\geq0$.
 Let $A:=H^0(G,\calO_G)$ be the Hopf algebra associated to $G$.
 \begin{enumerate}
     \item $A$ is semi-simple if and only if 
     $G$ is \'etale.
     \item $A$ is co-semi-simple if and only if $G$ is linearly
     reductive.
     \item $A$ is bi-semi-simple if and only if $G$ is 
     of length prime to $p$.
 \end{enumerate}
\end{Proposition}

\begin{proof}
We already established Claim (2) above.

If $G$ is \'etale, then $A\cong k[G_{\rm abs}]^*$,
where $G_{\rm abs}:=G(k)$ is the abstract group associated to $G$.
Since group rings are co-semi-simple, $A$ is semi-simple.
Conversely, if $A$ is semi-simple, then $\langle \varepsilon,\int_H^r\rangle\neq0$
by Maschke's theorem for Hopf algebras
(see, for example, \cite[Section 12.3.1]{Lorenz} for notation, statement, and proof), 
which implies that $A$ is a separable $k$-algebra,
which implies that $G$ is \'etale over $k$.
This establishes Claim (1).

Of course, (3) follows immediately from (1) and (2), but we can also give an independent
proof: Being an involutive Hopf algebra, $A$ is bi-semi-simple
if and only if $p$ does not divide $\dim_k A$, see 
\cite[Corollary 3.2]{EG} or \cite[Corollary 2.6]{LR}.
The latter is equivalent to $G$ being of length prime to $p$.
This establishes Claim (3).
\end{proof}

\subsection{Adjoint representation}
\label{subsec: adjoint}
Let $G_{\rm abs}$ be a finite group and let $k$ be an algebraically closed field
of characteristic $p\geq0$
($p$ may or may not divide the order of $G_{\rm abs}$).
Then, the action of $G_{\rm abs}$ on itself by conjugation is a permutation representation 
and we denote by 
$$
 \rho_{\rm ad} \,:\, G_{\rm abs} \,\to\, \GL(V_{\rm ad})
$$
the associated $k$-linear representation.

\begin{Remark}
To be more precise:
Depending on whether one considers $x\mapsto gxg^{-1}$ or $x\mapsto g^{-1}xg$,
one should speak about \emph{left} or \emph{right} adjoint actions and representations.
For our discussion, this is not important, as long as one chooses one of them
and stays with it.
\end{Remark}

\begin{Remark}
\label{rem: adjoint rep group}
Let us recall a couple of general results about $\rho_{\rm ad}$, which are
well-known (I thank Frank Himstedt for explaining them to me):
\begin{enumerate}
\item Let $C$ be the set of conjugacy classes of $G_{\rm abs}$ and let $\{g_c\}_{c\in C}$ be a system
of representatives. 
Then, we have an isomorphism of $k$-linear representations
$$
 \rho_{\mathrm{ad}} \,\cong\, \bigoplus_{c\in C}\, \mathrm{Ind}_{C(g_c)}^{G_{\rm abs}}\, \mathbb{1},
$$
where $C(g_c)$ denotes the centraliser of $g_c\in G_{\rm abs}$, where $\mathbb{1}$ denotes
the one-dimensional trivial representation, and where $\mathrm{Ind}$ denotes
induction from a subgroup.
See, for example, \cite[page 171 and Exercise 1.5]{NT}.
\item Let $g\in G_{\rm abs}$ and let
$$
 V\,:=\,\mathrm{Ind}_{C(g)}^{G_{\mathrm{abs}}}\, \mathbb{1} \,\cong\,\bigoplus_i V_i
$$
be a decomposition into indecomposable representations,
which is unique up to isomorphism and numbering.

There is precisely one summand, say $V_{i_1}$, that 
has a subrepresentation that is isomorphic to $\mathbb{1}$.
Moreover, this subrepresentation is unique.
Also, there  is precisely one summand, say $V_{i_2}$, that 
has a quotient representation that is isomorphic to $\mathbb{1}$.
Moreover, this quotient representation is unique.
Then, we have $i_1=i_2$ and set
    $$
       S(V) \,:=\, V_{i_1}\,=\, V_{i_2},
    $$
which is called the \emph{Scott representation} of $V$.
See, for example, \cite[page 296 and Theorem 8.4]{NT}. 
\end{enumerate}

In particular, the dimension of the largest trivial subrepresentation (resp. quotient representation) 
of $\rho_{\rm ad}$ is equal to the number of conjugacy classes of $G_{\rm abs}$.
\end{Remark}

If $A$ is a finite dimensional Hopf algebra over $k$, then there is an \emph{adjoint action}
of $A$ on itself (in fact, a \emph{left} adjoint action and a \emph{right} adjoint action), 
see, for example, \cite[Definition 3.4.1]{Montgomery}.
We will write
$$
 {\rm ad}\,:\,A \,\to\, {\rm End}(A)
$$
or simply ${}^{\rm ad} A$ for this representation.

\begin{Remarks}
\quad
\begin{enumerate}
\item
If $G_{\mathrm{abs}}$ is a finite group, then the adjoint representation
of $H:=k[G_{\mathrm{abs}}]$ can be identified with the $k$-linear extension
from $G_{\mathrm{abs}}$ to $H$
of the conjugation action of  $G_{\rm abs}$ on itself.
\item Let $G$ be a finite group scheme over $k$ and let $A:=H^0(G,\OO_G)$ be the 
$A:=H^0(G,\OO_G)$.
There is an adjoint representation 
$$
 \rho_{\rm ad} \,:\,G\,\to\,\GL(V_{\rm ad}),
$$
which corresponds to the adjoint representation of the \emph{dual} Hopf algebra $A^*$.
\begin{enumerate}
\item If $G_{\rm abs}$ is a finite group and $G$ is the constant 
group scheme associated to it, then $A$
 is isomorphic to the dual of
$k[G_{\rm abs}]$ equipped with its usual Hopf algebra structure.
This shows that $\rho_{\rm ad}$ should be defined
via the adjoint representation of $A^*$ rather than $A$.
\item The representation $\rho_{\rm ad}$ should not be confused with the adjoint representation of $G$ 
on its Lie algebra, see for example, \cite[page 100, Exercise 13]{Waterhouse}.
The latter can be identified with a subquotient of $\rho_{\rm ad}$.
\end{enumerate}
\end{enumerate}
\end{Remarks}

\subsection{Quantum doubles}
\label{subsec: quantum doubles}
If $A$ is a finite-dimensional Hopf algebra over a field $k$,
then Drinfeld \cite{Drinfeld} defined a Hopf algebra 
$D(A):=(A^{\rm op})^* \bowtie A$, called
the \emph{quantum double} or \emph{Drinfeld double},
where the bicrossed product structure is defined using 
the co-adjoint representation of $A$ on $A^*$ and the co-adjoint
representation of $A^*$ on $A$, 
see also \cite[Definition 10.3.1]{Montgomery}.

\begin{Remarks}
\label{rem: quantum elementary}
Let $A$ be a finite-dimensional Hopf algebra over $k$
and let $D(A):=(A^{\mathrm{op}})^* \bowtie A$ be its quantum double.
\begin{enumerate}
\item As $k$-vector space, $D(A)$ is of dimension $(\dim_kA)^2$.
\item $A$ is a Hopf subalgebra of $D(A)$ via $\varepsilon\bowtie A$
and $(A^{\mathrm{op}})^*$ is a Hopf subalgebra of $D(A)$ via 
$(A^{\mathrm{op}})^*\bowtie 1$.
\item If $A$ is commutative and co-commutative, then we have
$A^{\mathrm{op}}=A$ and using the definition of the
quantum double, we obtain isomorphisms
$$ 
  D(A)\,\cong\,A^*\otimes_k A
  \,\cong\, D(A^*)
$$ 
of Hopf algebras over $k$, where the tensor product is 
the trivial tensor product of Hopf algebras.
If  we set $G:=\Spec A$, 
then we have $G^D\cong \Spec A^*$, where $-^D$ denotes
the Cartier dual group scheme, and we have
an isomorphism
$$
 \Spec D(A) \,\cong\, G\times_{\Spec k} G^D
$$
of finite and commutative group schemes over $k$.
\item In particular,
$D(A)$ is commutative if and only if $A$ is commutative and co-commutative.
\end{enumerate}
\end{Remarks}

The assertions on simplicity in the next proposition extend results of
Witherspoon \cite[Proposition 1.2]{WitherspoonRepDouble} from groups 
to group schemes.
They are also related to general semi-simplicity 
results of quasi-triangular Hopf algebras in positive characteristic 
due to Etingof and Gelaki \cite{EG} and in this form, they might be known to the experts.

\begin{Proposition}
\label{prop: quantum properties}
 Let $G$ be a finite group scheme over an algebraically closed
 field $k$ of characteristic $p\geq0$ and let 
 $A:=H^0(G,\calO_G)$ be the associated Hopf algebra.
 Then, the following  are equivalent
 
 \begin{tabular}{llll}
 \emph{(1)} & $D(A)$ is semi-simple  & \emph{(1')} & $D(A^*)$ is semi-simple \\  
 \emph{(2)} & $D(A)$ is co-semi-simple \qquad \qquad \ & \emph{(2')} & $D(A^*)$ is co-semi-simple \\
 \emph{(3)} & $A$ is bi-semi-simple & \emph{(3')} & $A^*$ is bi-semi-simple \\
 \emph{(4)} & $G$ is of length prime to $p$ 
 \end{tabular}
\end{Proposition}

\begin{proof}
The equivalences (1) $\Leftrightarrow$ (2) $\Leftrightarrow$ (3) and
(1') $\Leftrightarrow$ (2') $\Leftrightarrow$ (3') are shown in \cite[Corollary 10.3.13]{Montgomery}.
The equivalence (3) $\Leftrightarrow$ (3') is trivial and
finally, the equivalence (3) $\Leftrightarrow$ (4) was shown in Proposition \ref{prop: hopf simple}.
\end{proof}

\subsection{Extended adjoint representation}
\label{subsec: extended adjoint}
Given a Hopf algebra $A$, we defined the adjoint representation
${}^{\mathrm{ad}}A$ in Section \ref{subsec: adjoint} above.
There is a way to extend this to a representation 
$$
\mathrm{Ad} \,:\, D(A) \,\to\, {\rm End}(A)
$$
the \emph{extended adjoint representation} of $A$, which is 
denoted by ${}^{\mathrm{Ad}} A$.
For its definition, we refer to Zhu's article \cite{Zhu}, as well as \cite{CW} for
subsequent work on this representation.
We note that it
can be described as an induced representation
\begin{equation}
\label{eq: Jacoby}
{}^{\mathrm{Ad}}\,A \,\cong\, {\mathrm{Ind}}_A^{D(A)}\,\mathbb{1},
\end{equation}
where $\mathbb{1}$ denotes the trivial one-dimensional representation
and where we consider $A$ as a subalgebra of $D(A)$ via 
$\varepsilon\bowtie A$, see \cite{Burciu} or \cite[Section 4.3.4]{Jacoby}.

\begin{Example}
\label{ex: Jacoby completely reducible} 
 Let $G_{\rm abs}$ be a finite group
 and let $H:=k[G_{\rm abs}]$ be the group algebra equipped with
 its usual Hopf algebra structure.
 The elements $g\in G_{\rm abs}$ form a basis of $H$ as $k$-vector space
 and we denote by $\rho_g\in H^*$ the dual basis elements.
 \begin{enumerate}
 \item The multiplication of $D(H)$ is given by
 $$
   (\rho_h\bowtie g)\cdot(\rho_k\bowtie \ell) \,=\, \delta_{h,gkg^{-1}}\cdot \rho_h\bowtie g\ell,
 $$
 see, for example, \cite[Example 2.4.2]{Jacoby}.
 \begin{enumerate}
 \item The extended adjoint representation ${}^{\rm Ad} H$ 
 is given by
 $$
 (\rho_h\bowtie k)\,\mapsto\,
 (g\,\mapsto\, \delta_{h^{-1},kgk^{-1}}\,\cdot\,
 kgk^{-1}
 )
 $$
 see, for example, \cite[Example 2.5.3]{Jacoby}.
 \item The unit of $H^*$ is $\varepsilon:=\sum_{g\in G_{\rm abs}}\rho_g$ and
 the restriction of ${}^{\mathrm{Ad}}H$ to $\varepsilon\bowtie H$ 
 is  ${}^{\rm ad}H$.
 Thus, the extended adjoint representation 
 extends the adjoint representation from an $A$-representation to 
 a $D(A)$-representation, whence the name.
 \item The extended adjoint representation ${}^{\rm Ad}H$  
 is semi-simple.
 More precisely, the $k$-subvector space of $H$ generated by
 the elements of a conjugacy class of $G_{\mathrm{abs}}$
 is a simple $D(H)$-subrepresentation
 of ${}^{\mathrm{Ad}}H$.
 This gives a bijection between conjugacy classes
 of $G_{\mathrm{abs}}$ and simple subrepresentations of 
 ${}^{\mathrm{Ad}}H$.
 By \cite[Lemma 4.3.2]{Jacoby}, this is also true if $p$ divides the order of 
 $G_{\mathrm{abs}}$.
 \end{enumerate}
 We refer to \cite{Burciu, Dijkgraaf, Gould, WitherspoonRepDouble} 
 for more results about the representation theory of $D(H)$
 and ${}^{\mathrm{Ad}}H$.
 \item 
 If we identify $H^{**}$ with $H$, then
 the multiplication of $D(H^*)$ is given by
  $$
   (g\bowtie \rho_h)\cdot(\ell\bowtie\rho_k) \,=\, 
   \delta_{h,k} \,\cdot\, (g\cdot \ell) \bowtie \rho_h
 $$
 and the extended adjoint representation ${}^{\rm Ad} (H^*)$ 
 is given by
 $$
 (h\bowtie \rho_k)\,\mapsto\,
 (\rho_g\,\mapsto\, \rho_{gh^{-1}}),
 $$
 which can be identified with the dual of the regular representation of $G_{\rm abs}$.
 The restriction of ${}^{\rm Ad}(H^*)$ to $1\bowtie H^*$ is trivial, which is equal
 to the adjoint representation $^{\rm ad}(H^*)$, which is also trivial.
 \end{enumerate}
\end{Example}

\begin{Example}
\label{ex: commutative extended adjoint}
Let $A$ be a finite-dimensional Hopf algebra over $k$ that is
commutative and co-commutative. 
\begin{enumerate}
\item The adjoint representations ${}^{\rm ad}A$ and ${}^{\rm ad}(A^*)$ are trivial.
\end{enumerate}
By Remark \ref{rem: quantum elementary}, we have isomorphisms
$$ 
  D(A)\,\cong\,A^*\otimes_k A
  \,\cong\, D(A^*)
$$ 
of Hopf algebras over $k$, where the tensor product is 
the trivial tensor product of Hopf algebras.
Thus, to give a representation $\rho:D(A)\to\mathrm{End}(V)$ is equivalent to 
giving two representations $\rho_1:A\to\mathrm{End}(V)$ and $\rho_2:A^*\to\mathrm{End}(V)$,
whose actions on $V$ commute.
\begin{enumerate}
\setcounter{enumi}{1}
\item Using the description \eqref{eq: Jacoby} of ${}^{\rm Ad}A$
as induced representation, we obtain ${}^{\rm Ad}A$ from the
the trivial $A$-action on $A$ (this is $\rho_1$) and 
the dual of the regular representation of $A^*$ (this is $\rho_2$).
\item Similarly, we obtain ${}^{\rm Ad}(A^*)$ from the
trivial $A^*$-action on $A^*$ and the 
dual of the regular representation of $A$.
\end{enumerate}
\end{Example}

\section{Conjugacy classes for finite group schemes}
\label{app: conjugacy class}

Let $G$ be a finite group scheme over an algebraically closed field 
$k$ of characteristic $p\geq0$.
In the second section of the appendix, we discuss several approaches toward
the notion of a \emph{conjugacy class} for $G$.
If $p=0$, then all of them lead to the same notion, namely, the familiar one.
All approaches look reasonable at first sight if $p>0$
and they lead to essentially ``the same'' answer if $G$ is 
of length prime to $p$.
In general however, they lead to different notions, all of which have
their merits and drawbacks.
This appendix serves as a motivation for Definition \ref{def: conjugacy class},
but the discussion and results may be interesting in themselves.

\subsection{First approach: via rational points}
Let $G(k)$ be the group of $k$-rational points.
We obtain an equivalence relation $\sim$ on this set
by defining $g_1\sim g_2$ if and only
if there exists $h\in G(k)$ such that $g_1=hg_2h^{-1}$.
The quotient $G(k)/\sim$ is the first candidate
for the set of conjugacy classes.

\begin{Example}
If $G$ is \'etale over $k$, then it is 
the constant group scheme associated
to the finite group $G_{\rm abs}:=G(k)$.
In this case, $G(k)/\sim$ coincides with the set of conjugacy
classes of $G_{\rm abs}$.
\end{Example}

If $G$ is \'etale, which is automatic if $p=0$,
then this approach is satisfactory.
However, if $p>0$, then the connected-\'etale sequence \eqref{connected etale}
induces a bijection
$$
 G(k) \,\cong\, G^{\et}(k),
$$
which is an isomorphism of finite groups.
In particular, $G(k)/\sim$ depends on the maximal \'etale quotient
$G^{\et}$ of $G$ only.
For example, in the extremal case where $G$ is connected, 
we have $G(k)=\{1\}$ and then, $G(k)/\sim$ consists of one element, 
and we do not gain much information about $G$.

Concerning functoriality: if $\varphi:G\to H$ is a homomorphism of
finite group schemes over $k$, then we have induced morphisms
$G(k)\to H(k)$, $G^{\et}\to H^{\et}$, and 
$G(k)/\sim\to H(k)/\sim$.

\subsection{Second approach: via representations and K-theory}
\label{subsec: second approach}
If $G_{\rm abs}$ is a finite group of order
prime to $p$, then the category $\Rep_k(G)$
of $k$-linear and finite-dimensional $G_{\rm abs}$-representations
is semi-simple. 
In this case, the number of isomorphism classes of simple representations 
is equal to the number of conjugacy classes.
More precisely, if $\rho$ is a $k$-linear and finite-dimensional
representation of $G_{\rm abs}$ and $g\in G_{\rm abs}$ is an element,
then 
$$
 (\rho,g) \,\mapsto\,{\rm Tr}\left(\rho(g)\right)\,\in\, k
$$ 
induces a pairing
between representations and conjugacy classes.
Since the character table of a finite group
is a quadratic and invertible matrix
(see, for example, \cite[Proposition I.7]{Serre}),
this pairing is non-degenerate and
one can think of conjugacy classes as being ``dual'' to 
simple representations.

This idea can be made precise as follows:
Let $F$ be a field, let $\mathrm{Cl}(G_{\mathrm{abs}})$ be the set of conjugacy 
classes of $G_{\mathrm{abs}}$, and let $F^{\mathrm{Cl}(G_{\mathrm{abs}})}$ be
the ring of class functions on $G_{\mathrm{abs}}$ with values in $F$.
Let $g$ be the order of $G_{\mathrm{abs}}$, let
$\zeta_g\in\CC$ be a primitive $g$.th primitive root of unity,
and assume that $F$ contains $\QQ(\zeta_g)$.
Since the characters of $G_{\mathrm{abs}}$ take values in $F$
(here, we use $\QQ(\zeta_g)\subseteq F$), we have injective ring homomorphisms
$$
 F \,\to\, F\otimes_\ZZ K_k(G_{\mathrm{abs}}) \,\stackrel{\gamma}{\longrightarrow}\, F^{\mathrm{Cl}(G_{\mathrm{abs}})}
$$
and $\gamma$ is an isomorphism, see \cite[Section 11.4]{Serre} and
the first paragraph of \cite[Section 9.1]{Serre}.
In particular, 
$$
\Spec\, F\otimes_\ZZ K_k(G_{\mathrm{abs}})
$$
is a finite set consisting of maximal ideals only and the topology is discrete.
Since $\gamma$ is an isomorphism, the cardinality of this set is equal to that
of $\mathrm{Cl}(G_{\mathrm{abs}})$.
This carries over to finite and linearly reductive schemes as follows.

\begin{Proposition}
\label{prop: conjugacy lift and bijection}
Let $G$ be a finite and linearly reductive group scheme over $k$
and let $G_{\mathrm{abs}}$ be the associated abstract group.
Let $g$ be the length of $G$, fix a $g$.th root of unity $\zeta_g\in\CC$, and
let $F$ be a field that contains $\QQ(\zeta_{g})$.
Then, there exista canonical bijection of sets
$$
\left\{\begin{array}{l}
\mbox{conjugacy classes of $G_{\rm abs}$}
\end{array}\right\}
\,\to\,
\Spec\, F\otimes_\ZZ K_k(G_{\rm abs})
$$
and homeomorphisms
$$
\Spec\, F\otimes_\ZZ K_k(G_{\rm abs})
\,\leftarrow\,
\Spec\, F\otimes_\ZZ K_{\overline{K}}(G_{\rm abs})
\,\leftarrow\,
\Spec\, F\otimes_\ZZ K_k(G).
$$
\end{Proposition}

\begin{proof}
If $c$ is a conjugacy class of $G_{\rm abs}$,
then $P_{0,c}$ (notation as in \cite[Section 11.4, Proposition 30]{Serre})
is an element of $\Spec F\otimes K_k(G_{\rm abs})$
and by loc. cit. this defines a bijection.
The isomorphisms of Corollary \ref{cor: Ktheory} induce homeomorphisms and thus,
bijections of spectra as stated.
\end{proof}

\begin{Remarks}
\label{rem: conjugacy classes}
\quad
\begin{enumerate}
\item The maximal abelian extension $\QQ\subset\QQ^{\rm ab}$ 
is generated by all roots of unity by the
Kronecker-Weber theorem.
Thus, if we have  $\QQ^{\mathrm{ab}}\subseteq F$, then we have a field
that works independent of the length of $G$.
In Definition \ref{def: conjugacy class}, we have chosen $F=\CC$
as this field may be more familiar than $\QQ^{\mathrm{ab}}$.

\item 
In \cite[Section 11.4]{Serre}, Serre described
$\Spec A\otimes_\ZZ K_k(G_{\rm abs})$, where $A=\ZZ[\zeta_g]$.
Using Corollary \ref{cor: Ktheory}, we obtain a homeomorphism
$$
 \Spec A\otimes_\ZZ K_k(G) \,\to\,\Spec A\otimes_\ZZ K_k(G_{\rm abs}).
$$
In particular, Serre's results from loc.cit. carry over to $A\otimes_\ZZ K_k(G)$.
For the purposes of this article, we are only interested
in the fibre over $0\in\Spec F$ with $F=\mathrm{Frac}(A)=\QQ(\zeta_g)$, 
that is, the Zariski open subset 
$\Spec F\otimes K_k(G)\subset\Spec A\otimes K_k(G)$.
\end{enumerate}
\end{Remarks}

\begin{Example}
 Let $G$ be the group scheme $\balpha_p$ or $\C_p$ 
 over the algebraically closed field $k$ of characteristic $p>0$.
 Then, $\Rep_k(G)$ is not semi-simple:
 the only simple $k$-linear representation of $G$
 is the trivial one-dimensional representation $\mathbb{1}$.
 Thus, $\mathbb{1}\mapsto1$ induces an isomorphism of rings
 $K_k(G)\cong\ZZ$ and
 $\Spec F\otimes_\ZZ K_k(G)$ consists of one point only.
 The approach to conjugacy classes in this subsection 
 may therefore lead to somewhat unexpected results if 
 $G$ is not linearly reductive.
\end{Example}

\begin{Proposition}
 Let $\varphi:G\to H$ be a morphism of finite and linearly reductive group schemes
 over $k$.
 Let $\varphi_{\rm abs}:G_{\rm abs}\to H_{\rm abs}$ be the induced homomorphism
 of their associated abstract groups.
 Let $F$ be a field that contains $\QQ(\zeta_g,\zeta_h)$, where $g$ (resp. $h$)
 denotes the length of $G$ (resp. $H$).
 \begin{enumerate}
 \item There maps $\varphi$ and $\varphi_{\mathrm{abs}}$ induce ring homomorphisms
 $K_k(H)\to K_k(G)$ and $K_{\overline{K}}(H_{\mathrm{abs}})\to K_{\overline{K}}(G_{\mathrm{abs}})$,
 respectively.
 We obtain a commutative diagram of continuous maps
 $$
 \xymatrix{
  \Spec F\otimes_\ZZ K_k(G) \ar[r]\ar[d] &\Spec F\otimes_\ZZ K_k(H)\ar[d]\\
  \Spec F\otimes_\ZZ K_{\overline{K}}(G_{\mathrm{abs}}) \ar[r] & \Spec F\otimes_\ZZ K_{\overline{K}}(H_{\mathrm{abs}}),
  }
 $$
 whose vertical arrows are the homeomorphisms 
from Proposition \ref{prop: conjugacy lift and bijection}.
\item Let  $G_{\mathrm{abs}}/\sim\to H_{\mathrm{abs}}/\sim$ be the map on conjugacy classes
induced by $\varphi_{\mathrm{abs}}$.
We obtain a commutative diagram of maps of sets
 $$
 \xymatrix{
  G_{\rm abs}/\sim \ar[r]\ar[d] &H_{\rm abs}/\sim\ar[d]\\
  \Spec F\otimes_\ZZ K_{\overline{K}}(G_{\mathrm{abs}}) \ar[r] & \Spec F\otimes_\ZZ K_{\overline{K}}(H_{\mathrm{abs}}),
  }
 $$
whose vertical maps are the bijections from  Proposition \ref{prop: conjugacy lift and bijection}.
\end{enumerate}
\end{Proposition}

\begin{proof}
Clearly, $\varphi$ induces a morphism $K_k(H)\to K_k(G)$ of rings since every 
$H$-representation becomes a $G$-representation via $\varphi$.
Similarly, $\varphi_{\mathrm{abs}}$ induces a ring homomorphism 
$\varphi_{K,\mathrm{abs}}\,:\,K_{\overline{K}}(H_{\mathrm{abs}})\to K_{\overline{K}}(G_{\mathrm{abs}})$.
We leave it to the reader to check the compatibility of these maps with the homeomorphisms
of Proposition \ref{prop: conjugacy lift and bijection}.

To check commutativity of the second diagram, let $g\in G_{\rm abs}$.
With the notations and definitions of \cite[Section 11.4]{Serre}, it is easy to see that
we have 
$$
 P_{0,\varphi_{\mathrm{abs}}(g)} \,=\,\varphi_{K,\mathrm{abs}}^{-1}(P_{0,g}) \,=\,
 \varphi_{K,\mathrm{abs}}^\sharp(P_{0,g}),
$$
where 
$$
\varphi_{K,\mathrm{abs}}^\sharp \,:\,\Spec F\otimes K_{\overline{K}}(G_{\mathrm{abs}}) \,\to\, 
\Spec F\otimes K_{\overline{K}}(H_{\mathrm{abs}})
$$
is the induced map on spectra.
Since the image of the conjugacy class $[g]$ of $G_{\mathrm{abs}}$ is the
conjugacy class $[\varphi_{\mathrm{abs}}(g)]$ of $H_{\mathrm{abs}}$, the assertion follows.
\end{proof}

\subsection{Third approach: the scheme of conjugacy classes}
\label{subsec: scheme of conjugacy classes}
Just as group schemes generalise the notion of a group,
one could try to replace the \emph{set} of conjugacy classes 
by a suitable notion of \emph{scheme} of conjugacy classes.
More precisely, let $G$ be a finite group scheme over $k$
and let $\underline{\mathrm{Aut}}(G)$ be the automorphism group
scheme of $G$.
For every scheme $T\to\Spec k$, we have the set $G(T)$ of $T$-valued
points of $G$ and a conjugation action of $G(T)$ on $G(T)$.
This induces a morphism of schemes $G\to\underline{\mathrm{Aut}}(G)$.
We will say that two elements of $G(T)$ are \emph{equivalent} if they differ
by such an automorphism and we denote the resulting equivalence relation by
$\sim$.

We obtain a functor from the category of schemes over $k$ to sets
$$
\begin{array}{cccccc}
\underline{{\rm Conj}}_G &:& ({\rm Schemes}/k) &\to& ({\rm Sets})\\
&& T &\mapsto& G(T)/\sim &.
\end{array}
$$
This functor should somehow represent the
conjugacy classes of $G$ in the sense of schemes.

\begin{Proposition}
 Let $G$ be a finite group scheme over $k$.
 Then, the functor $\underline{{\rm Conj}}_G$ is representable by a 
 scheme ${\rm Conj}_G$, which is finite over $\Spec k$.
\end{Proposition}

\begin{proof}
Set $V:=H^0(G,\calO_G)$ and 
let $\rho_{{\rm ad}}:G\to \GL(V)$ be the adjoint representation of $G$
(see Appendix \ref{subsec: adjoint}).
Then, the functor $\underline{\mathrm{Conj}}_G$ can be rephrased
as the functor that associates to each $T\to\Spec k$ the quotient
of $V\times\OO_T$ modulo $G(T)$.
This amounts to representing the quotient $V/G$ by a scheme.
Since $V$ is a vector space and thus, can be identified with an affine scheme,
and since $G$ is a finite group scheme, this
quotient is representable by a scheme.
In fact, it is representable by the spectrum of the invariant
ring $\Spec V^G$.
\end{proof}

\begin{Definition}
 $\mathrm{Conj}_G$ is called the \emph{scheme of conjugacy classes}
 of $G$.
\end{Definition}

\begin{Remark}
\label{rem: adjoint}
 If $\rho_{\mathrm{ad}}:G\to \GL(H^0(G,\calO_G))$ is the adjoint representation, 
 then the previous proof shows that the length of $\mathrm{Conj}_G$ over $\Spec k$ 
 is equal to the dimension
 of the maximal trivial subrepresentation of $\rho_{\mathrm{ad}}$.
 If $G_{\rm abs}$ is a finite group or if $G$ is a finite and \'etale group scheme 
 over $\Spec k$, then we gave an explicit description of this maximal trivial 
 subrepresentation of $\rho_{\mathrm{ad}}$ in Remark \ref{rem: adjoint rep group}.
\end{Remark}

\begin{Examples}
\label{ex: conj scheme}
\quad 
\begin{enumerate}
\item If $G$ is \'etale over $k$, then $G$ is isomorphic to the constant
group scheme associated to $G_{\mathrm{abs}}:=G(k)$.
In this case, $\mathrm{Conj}_G$ is a disjoint union
of copies of $\Spec k$, one copy for each conjugacy class of the
abstract group $G_{\mathrm{abs}}$.
\item If $G$ is commutative, then the adjoint representation is trivial
and thus, $\mathrm{Conj}_G$ is isomorphic to the scheme underlying $G$.
For example, if $G=\balpha_p$ or $G=\bmu_p$, then 
$\mathrm{Conj}_G$ is a non-reduced scheme of length $p$
with reduction $(\mathrm{Conj}_G)_{\mathrm{red}}\cong\Spec k$.
\end{enumerate}
\end{Examples}

In characteristic zero, all finite group schemes are \'etale and 
Case (1) applies.
However, in characteristic $p>0$, the second example shows that 
for non-reduced group schemes $\mathrm{Conj}_G$ may also be non-reduced.
The reduction $(\mathrm{Conj}_G)_{\mathrm{red}}$ is related to our
first approach to conjugacy classes: 

\begin{Proposition}
\quad
 \begin{enumerate}
 \item A morphism $G\to H$ of finite group schemes over $k$ induces
a morphism $\mathrm{Conj}_G\to \mathrm{Conj}_H$ of schemes over $k$.
 \item Assume $p>0$ and let $G\to G^{\et}$ be the maximal \'etale quotient
 of the finite group scheme $G$ over $k$.
 Then, reduction identifies $G_{\mathrm{red}}$ with $G^{\et}$ and
 the induced natural inclusion $G^{\et}\to G$ induces an isomorphism
 $$
  \mathrm{Conj}_{G^{\et}} \,\to\, \left( \mathrm{Conj}_G \right)_{\mathrm{red}}
 $$
 of schemes over $k$.
 In particular, $(\mathrm{Conj}_G)_{\mathrm{red}}$ is 
 a disjoint union of copies of $\Spec k$ with one copy 
 for each conjugacy class of the abstract group
 $G_{\mathrm{abs}}:=G(k)=G^{\et}(k)$.
 \end{enumerate}
\end{Proposition}

\begin{proof}
We first prove Claim (1).
If $G_{\rm abs}\to H_{\rm abs}$ is a homomorphism of groups,
then we get a well-defined induced map of conjugacy classes.
Thus, if $G\to H$ is as in (1) and if $T$ is a scheme over $k$,
then we get a well-defined map $(G(T)/\sim)\to (H(T)/\sim)$.
This induces a morphism of functors
$\underline{\mathrm{Conj}}_G\to\underline{\mathrm{Conj}}_H$
and thus, a morphism of schemes
$\mathrm{Conj}_G\to\mathrm{Conj}_H$.

Let $G$ be as in Claim (2).
Let $G_{\mathrm{red}}\to G$ be the reduction, which is a morphism
of group schemes.
We thus obtain canonical homomorphisms
$G^{\et}\to G\to G^{\et}$ of group schemes over $k$, which
(by Claim (1)) induce morphisms of their associated schemes of 
conjugacy classes
$$
  \mathrm{Conj}_{G^\et} \,\to\, \mathrm{Conj}_G \,\to\, \mathrm{Conj}_{G^\et}.
$$
The composition is the identity. 
Since $\mathrm{Conj}_{G^{\et}}$ is reduced, we obtain a factorisation
\begin{equation}
\label{eq: conj}
  \mathrm{Conj}_{G^\et} \,\to\, \left(\mathrm{Conj}_G\right)_{\mathrm{red}} \,\to\, 
  \mathrm{Conj}_{G^\et}.
\end{equation}
All these schemes are reduced and finite over the algebraically closed
field $k$.
Thus, all of them are finite disoint unions of copies of $\Spec k$.
To prove that the morphisms in \eqref{eq: conj} are isomorphisms, it
suffices to check that the induced maps on $k$-rational points
are bijections.
This follows easily from the fact that
the maps $G_{\mathrm{red}}(k)\to G(k)\to G^{\et}(k)$ are 
bijections.
\end{proof}

If $G$ is moreover linearly reductive, 
then the length of $\mathrm{Conj}_G$ is related to the approach to conjugacy classes
from Section \ref{subsec: second approach}.

\begin{Proposition}
 Let $G$ be a finite and linearly reductive group scheme over $k$,
 let $G_{\mathrm{abs}}$ be the abstract group associated to $G$, 
 and let $F$ be a field as in Proposition \ref{prop: conjugacy lift and bijection}.
 Then,  
 $$
  \mathrm{length}_k\, \mathrm{Conj}_G \,=\, | \Spec F\otimes_\ZZ K_k(G) |
 $$
 and this length agrees with the number of conjugacy classes
 of $G_{\mathrm{abs}}$.
\end{Proposition}

\begin{proof}
By Remark \ref{rem: adjoint}, the length of $\mathrm{Conj}_k(G)$
is equal to the dimension of the 
largest trivial subrepresentation of $\rho_{\mathrm{ad}}$.
Since ${\rm Rep}_k(G)$ is semi-simple, this is the same as the
multiplicity of $\mathbb{1}$ in $\rho_{\mathrm{ad}}$.
Using the isomorphism from Proposition \ref{prop: liftingrepresentation},
this multiplicity is the same as the multiplicity of $\mathbb{1}$
in the adjoint representation of $G_{\mathrm{abs}}$.
This latter multiplicity is equal to the number of conjugacy classes of $G_{\mathrm{abs}}$,
see, for example, Remark \ref{rem: adjoint rep group}.
By Proposition \ref{prop: conjugacy lift and bijection},
this number is equal to the cardinality of $\Spec F\otimes_\ZZ K_k(G)$.
\end{proof}

\subsection{Fourth approach: via lifting to characteristic zero}
Another idea might be to use lifting to characteristic zero:
Let $k$ be an algebraically closed field of characteristic $p>0$, 
let $W(k)$ be the ring of Witt vectors of $k$, and let $K$ be field 
of fractions of $W(k)$.

\begin{Example}[Mumford--Oort]
Let $\varphi:\bmu_p\to\underline{\mathrm{Aut}}(\balpha_p)\cong\GG_m$ be a non-trivial
homomorphism and let $G:=\balpha_p\rtimes_\varphi\bmu_p$ be the
corresponding semi-direct product group scheme.
Then, $G$ is a non-commutative group scheme of length $p^2$ over $k$.
Oort and Mumford \cite[Introduction, Example (-B)]{MumfordOort}
(but see also \cite[page 266]{Oort}) showed that there does \emph{not}
exist a lift of $G$ to any extension of $W(k)$.
\end{Example}

In particular, one \emph{cannot} define
conjugacy classes by first lifting $G$ over some possibly ramified extension of
$W(k)$, then passing to the geometric generic fibre of the lift, which would
be the constant group associated to an abstract finite group $G_{\mathrm{abs}}$,
and finally use the conjugacy classes of $G_{\mathrm{abs}}$ as a replacement
for the conjugacy classes of $G$.
The reason is simply that lifts may not exist to start with.

In the following two cases, lifts do exist and we leave the straightforward
proofs of our assertions to the reader:

\begin{Remark}
Assume that $G$ is \'etale over $k$.
Then, there exists a unique flat lift of $G$ over $W(k)$, which is 
the constant group scheme associated to
$G_{\mathrm{abs}}:=G(k)$ over $W(k)$.
The constant scheme associated to the set 
of conjugacy classes of $G_{\mathrm{abs}}$
is the unique a flat lift of the scheme $\mathrm{Conj}_G$ over $W(k)$,
which is  \'etale over $W(k)$.
More precisely, it is the constant scheme associated to the
set of conjugacy classes of $G_{\mathrm{abs}}$ over $W(k)$.
\end{Remark}

\begin{Remark}
Assume that $G$ is linearly reductive.
Let $G_{\mathrm{abs}}$ abstract group associated to $G$
and recall that their representation categories are equivalent
by  Proposition \ref{prop: liftingrepresentation}.
Proposition \ref{prop: conjugacy lift and bijection}
shows that Definition \ref{def: conjugacy class}
is compatible with $G_{\rm abs}$. 

If $\calG\to \Spec W(k)$ is a flat lift of $G$ over $W(k)$, for example, the canonical lift
$\calG_{\mathrm{can}}$, then
$$
\Spec\, H^0(\calG,\calO_{\calG})^{\calG}\,\to\,\Spec W(k),
$$
where the invariants are taken with respect to the adjoint representation,
is a flat lift of $\mathrm{Conj}_G$ over $W(k)$, whose geometric generic
fibre is the constant scheme associated to the set of conjugacy classes of $G_{\mathrm{abs}}$.
(Flatness follows from the fact that the special and the geometric generic
fibre have the same length by Remark \ref{rem: adjoint}.)

The length of $\mathrm{Conj}_G$ is equal to the number
of conjugacy classes of $G_{\rm abs}$.
The length of $\mathrm{Conj}_G$ is at least
the number of $k$-rational points $\mathrm{Conj}_G(k)$
and in general not equal, since 
$\mathrm{Conj}_G$ may not be \'etale over $k$.
\end{Remark}

If $G$ is of length prime to $p$, then it is \'etale and linearly reductive
and in this case, all previous approaches yield essentially the same
notion of conjugacy class.
On the other hand, the examples of the above subsections show that
if $G$ is \'etale of length divisible by $p$ or if $G$ is linearly reductive but not
\'etale, then the various approaches of the above subsections usually
lead to different notions of conjugacy classes.

\subsection{Fifth approach: via adjoint representation}
Let $G_{\rm abs}$ be a finite group and let $\rho_{\mathrm{ad}}:G_{\mathrm{abs}}\to\GL(V_{\mathrm{ad}})$
be its adjoint representation over $k$.
By Remark \ref{rem: adjoint rep group}, the dimension of 
the largest trivial subrepresentation of $\rho_{\mathrm{ad}}$ 
is equal to the dimension of largest trivial quotient
representation  of $\rho_{\mathrm{ad}}$ (even if $p$ divides the order of $G_{\mathrm{abs}}$) and
these dimensions are equal to the number of conjugacy classes of $G_{\mathrm{abs}}$.
Unfortunately, these trivial sub- or quotient representations do not admit
\emph{canonical} decompositions into one-dimensional subspaces.
Thus, there is no canonical bijection between the conjugacy classes of $G_{\mathrm{abs}}$
with trivial sub- or quotient representations of $\rho_{\mathrm{ad}}$.

\begin{Remark}
Let $H:=k[G_{\mathrm{abs}}]$ be the group algebra with its usual Hopf algebra structure.
The set of group-like elements of $H$ recovers the group $G_{\mathrm{abs}}$,
see Example \ref{ex: group-like}.
This suggests to use the adjoint representation ${}^{\mathrm{ad}}H$ together with
the Hopf algebra structure of $H$ - in particular, its co-algebra structure -
to define a useful notion of conjugacy class.

As we have seen (somewhat implicitly) in Section \ref{subsec: scheme of conjugacy classes}
above, this works:
Let $G:=\Spec H^*$ be the constant group scheme associated to $G_{\mathrm{abs}}$.
Then, the topological space underlying $\mathrm{Conj}_G$ is the set of group-like elements of the 
Hopf algebra $H$ modulo the adjoint representation. 
This set can be identified with the set of conjugacy classes of $G_{\mathrm{abs}}$.
\end{Remark}

Unfortunately, the previous remark does not carry over to finite group schemes 
that are not \'etale, as the following example shows.

\begin{Example}
Let $G=\bmu_p$ or $G=\balpha_p$ over the algebraically closed field $k$
of characteristic $p>0$.
\begin{enumerate}
\item The adjoint representation $\rho_{\mathrm{ad}}$ of $G$ is trivial of dimension $p$,
but there is no canonical way to decompose it into one-dimensional subspaces.
This would suggest to have $p$ conjugacy classes, but without being able to distinguish them.
\item
Moreover, there is only one group-like element in the Hopf algebra $H^0(G,\OO_G)^*$
and the adjoint representation on this element is trivial as well.
This would suggest to have one conjugacy class only,
see also Example \ref{ex: conj scheme}.(2). 
\end{enumerate}
\end{Example}

\subsection{Sixth approach: via extended adjoint representation}
Given a finite-dimensional Hopf algebra $A$ over a field $k$,
we have the adjoint representation ${}^{\rm ad}A$, see Appendix \ref{subsec: adjoint}.
In Appendix \ref{subsec: quantum doubles}, we recalled the quantum double
$D(A)=(A^{\mathrm{op}})^*\bowtie A$.
In Appendix \ref{subsec: extended adjoint}, we recalled that ${}^{\mathrm{ad}}A$
can be extended to a representation of $D(A)$ on $A$, the
extended adjoint representation ${}^{\mathrm{Ad}}A$.

Generalising work of Witherspoon \cite{Witherspoon}, Cohen and Westreich \cite{CW}
defined the set of conjugacy classes of $A$ for a finite-dimensional 
semi-simple Hopf algebra over $\CC$ to be the set of 
simple subrepresentations of ${}^{\mathrm{Ad}}A$.
In this context, we also refer to the work of Jacoby \cite{Jacoby} and Zhu \cite{Zhu}.

If $G$ is a finite group scheme over $k$ with Hopf algebra $A:=H^0(G,\OO_G)$, then
$D(A)$ is semi-simple if and only if $D(A^*)$ is semi-simple if and only if
$G$ is of length prime to $p$, see Proposition \ref{prop: quantum properties}.
In particular, the (extended) adjoint representations
of $A$ or $A^*$ may not be semi-simple.
We will now study these representations and their relation to conjugacy classes.

\begin{Example}
Let $G$ be \'etale over $k$.
Then it is the constant group scheme
associated to $G_{\mathrm{abs}}:=G(k)$.
If $A:=H^0(G,\OO_G)$ is the associated Hopf algebra, then we have
$A^*\cong k[G_{\mathrm{abs}}]$.
As seen in Example \ref{ex: Jacoby completely reducible},
the extended adjoint representation ${}^{\mathrm{Ad}}(A^*)$ is semi-simple
and there is a natural bijection between conjugacy classes of $G_{\mathrm{abs}}$
and simple subrepresentations of ${}^{\mathrm{Ad}}(A^*)$.
We stress that this is also true if $p$ divides the order of $G_{\mathrm{abs}}$.
\end{Example}

\begin{Example}
Let $G$ be a finite and linearly reductive group scheme over $k$, let $G_{\mathrm{abs}}$
be the associated abstract finite group, and let $A:=H^0(G,\OO_G)$ be the associated Hopf
algebra.
Since $A$ is commutative, ${}^{\mathrm{ad}}A$ is trivial.
Using the description \eqref{eq: Jacoby} of ${}^{\mathrm{Ad}}A$ as an induced representation,
we can identify it with the dual of the regular representation 
$$
\rho_{\mathrm{reg}}^\vee\,:\, A^*\,\to\, \mathrm{End}(A)
$$
together with the trivial representation of $A$.
By Proposition \ref{prop: liftingrepresentation},
representations of $A^*$ can be identified with representations
of $\CC[G_{\mathrm{abs}}]$ and thus, representations of $G_{\mathrm{abs}}$.
Thus, we can decompose ${}^{\rm Ad}A$ like the dual of the regular representation 
of $G_{\mathrm{abs}}$:
the number of isotypical components of this representation is equal to
the number of conjugacy classes of $G_{\mathrm{abs}}$.
Thus, one can think of the conjugacy classes of $G_{\mathrm{abs}}$ as
being ``dual'' to these isotypical components similarly to
Appendix \ref{subsec: second approach}.
\end{Example}

The upshot of this discussion is the following:
Let $G$ be a finite group scheme over $k$ and let $A:=H^0(G,\OO_G)$ be the associated
Hopf algebra.
\begin{enumerate}
\item If $G$ is \'etale over $k$, then simple subrepresentations of ${}^{\mathrm{Ad}}(A^*)$ 
are in bijection with conjugacy classes of $G_{\mathrm{abs}}:=G(k)$.
\item  If $G$ is linearly reductive, then the isotypical components of ${}^{\mathrm{Ad}}A$ give a 
reasonably good definition for the ``dual'' of a conjugacy class.
\end{enumerate}
The following example shows that in general, neither ${}^{\mathrm{Ad}}A$ nor ${}^{\mathrm{Ad}}(A^*)$
leads to  a good approach toward the notion of a conjugacy class - at least none that is 
better than the ones already discussed.

\begin{Example}
Let $G$ be a finite and commutative group scheme over $k$ and let
$A:=H^0(G,\OO_G)$ be the associated Hopf algebra.
As a consequence of Example \ref{ex: commutative extended adjoint}, we have the
following.
\begin{enumerate}
\item If $G=\bmu_p$, then we have $A=k[\C_p]$.
Thus, ${}^{\mathrm{Ad}}A$ splits into the direct sum of $p$ pairwise non-isomorphic 
one-dimensional representations,
which correspond to the characters of $G$.
On the other hand, ${}^{\mathrm{Ad}}(A^*)$ is a non-trivial 
successive $p$-fold and non-split extension of the trivial representation $\mathbb{1}$,
whose semi-simplification is trivial of dimension $p$.
\item If $G=\C_p$, then we obtain the same as before with the r\^{o}les of
$A$ and $A^*$ interchanged.
\item If $G=\balpha_p$, then $A\cong A^*$ and 
${}^{\mathrm{Ad}}A\cong{}^{\mathrm{Ad}}(A^*)$
is a successive $p$-fold and non-split extension of the trivial
representation $\mathbb{1}$.
\end{enumerate} 
\end{Example}

\end{document}